\documentclass[10pt]{article}
\usepackage{amsthm}
\usepackage{amsmath}
\usepackage{amsxtra}
\usepackage{amssymb}
\usepackage{enumitem}
\usepackage{accents}
\usepackage{mathrsfs}
\usepackage{turnstile}
\usepackage{thmtools}
\usepackage{IEEEtrantools}
\usepackage[colorlinks=true,linkcolor=blue]{hyperref}
\usepackage{tikz}
\usepackage{url}
\usepackage{bbm}

\swapnumbers
\declaretheorem[name=Definition,style=definition,qed=$\dashv$,
numberwithin=section]{dfn}
\declaretheorem[name=Definition,style=definition,numbered=no,qed=$\dashv$]{dfn*}
\declaretheorem[name=Definition,style=definition,numbered=no]{dfnnoqed*}

\declaretheorem[name=Theorem,style=plain,sibling=dfn]{tm}
\declaretheorem[name=Theorem,style=plain,numbered=no]{tm*}

\declaretheorem[name=Lemma,style=plain,sibling=dfn]{lem}
\declaretheorem[name=Corollary,style=plain,sibling=dfn]{cor}
\declaretheorem[name=Corollary,style=plain,numbered=no]{cor*}
\declaretheorem[name=Remark,style=definition,sibling=dfn]{rem}
\declaretheorem[name=Fact,style=definition,sibling=dfn]{fact}
\swapnumbers
\declaretheoremstyle[headfont=\scshape]{claimstyle}
\declaretheorem[name=Claim,style=claimstyle]{clm}
\declaretheorem[name=Claim,style=claimstyle]{clmtwo}

\declaretheorem[name=Claim,style=claimstyle]{clmfour}
\declaretheorem[name=Claim,style=claimstyle]{clmfive}
\declaretheorem[name=Claim,style=claimstyle]{clmsix}
\declaretheorem[name=Claim,style=claimstyle]{clmseven}
\declaretheorem[name=Claim,style=claimstyle,numbered=no]{clm*}

\declaretheorem[name=Subclaim,style=claimstyle,numberwithin=clmsix]{sclmsix}

\declaretheorem[name=Subclaim,style=claimstyle,numbered=no]{sclm*}

\declaretheorem[name=Subsubclaim,style=claimstyle,numbered=no]{ssclm*}

\declaretheoremstyle[headfont=\scshape]{casestyle}
\declaretheorem[name=Assumption,style=casestyle]{ass}

\declaretheorem[name=Case,style=casestyle]{case}
\declaretheorem[name=Case,style=casestyle]{casetwo}
\declaretheorem[name=Case,style=casestyle]{casethree}
\declaretheorem[name=Case,style=casestyle]{casefour}

\newcommand{\QQ}{\mathbb Q}
\newcommand{\RR}{\mathbb R}
\newcommand{\PP}{\mathbb P}
\newcommand{\sub}{\subseteq}
\newcommand{\cross}{\times}
\newcommand{\all}{\forall}
\newcommand{\ex}{\exists}
\newcommand{\om}{\omega}
\newcommand{\pow}{\mathcal{P}}
\newcommand{\OR}{\mathrm{OR}}
\newcommand{\Hull}{\mathrm{Hull}}

\newcommand{\Tt}{\mathcal{T}}
\newcommand{\Ss}{\mathcal{S}}
\newcommand{\Uu}{\mathcal{U}}

\newcommand{\rg}{\mathrm{rg}}
\newcommand{\dom}{\mathrm{dom}}
\newcommand{\ins}{\trianglelefteq}
\newcommand{\nins}{\ntrianglelefteq}
\newcommand{\pins}{\triangleleft}

\newcommand{\crit}{\mathrm{cr}}

\newcommand{\rest}{\!\upharpoonright\!}
\newcommand{\com}{\circ}
\newcommand{\lh}{\mathrm{lh}}
\newcommand{\Ult}{\mathrm{Ult}}
\newcommand{\Ebar}{{\bar{E}}}
\newcommand{\Fbar}{{\bar{F}}}
\newcommand{\sats}{\models}
\newcommand{\elem}{\preccurlyeq}
\newcommand{\J}{\mathcal{J}}
\newcommand{\PS}{\mathrm{PS}}

\newcommand{\AC}{\mathrm{AC}}

\newcommand{\HOD}{\mathrm{HOD}}
\newcommand{\HC}{\mathrm{HC}}
\newcommand{\ZFC}{\mathrm{ZFC}}

\newcommand{\wlim}{\mathrm{wlim}}
\newcommand{\es}{\mathbb{E}}
\newcommand{\ebar}{{\bar{e}}}
\newcommand{\jbar}{{\bar{j}}}
\newcommand{\qbar}{{\bar{q}}}
\newcommand{\tbar}{{\bar{t}}}
\newcommand{\mubar}{{\bar{\mu}}}
\newcommand{\taubar}{{\bar{\tau}}}
\newcommand{\gammabar}{{\bar{\gamma}}}
\newcommand{\thetabar}{{\bar{\theta}}}
\newcommand{\kappabar}{{\bar{\kappa}}}
\newcommand{\Pbar}{{\bar{P}}}

\newcommand{\pibar}{\bar{\pi}}
\newcommand{\Hbar}{{\bar{H}}}
\newcommand{\Qbar}{{\bar{Q}}}
\newcommand{\Ubar}{{\bar{U}}}
\newcommand{\core}{\mathfrak{C}}
\newcommand{\her}{\mathcal{H}}
\newcommand{\pred}{\mathrm{pred}}
\newcommand{\un}{\cup}
\newcommand{\id}{\mathrm{id}}
\newcommand{\sq}{\mathrm{sq}}
\newcommand{\conc}{\ \widehat{\ }\ }
\newcommand{\bfSigma}{\undertilde{\Sigma}}
\newcommand{\rSigma}{\mathrm{r}\Sigma}
\DeclareMathOperator{\Th}{Th}
\DeclareMathOperator{\card}{card}
\DeclareMathOperator{\cof}{cof}
\newcommand{\xvec}{\vec{x}}
\newcommand{\OD}{\mathrm{OD}}
\newcommand{\bfrSigma}{\undertilde{\rSigma}}
\newcommand{\psub}{\subsetneq}
\newcommand{\cHull}{\mathrm{cHull}}
\newcommand{\minterm}{\mathrm{m}\tau}

\newcommand{\Mbar}{{\bar{M}}}
\newcommand{\lpole}{\left\lfloor}
\newcommand{\rpole}{\right\rfloor}
\newcommand{\univ}[1]{\lpole #1\rpole}
\newcommand{\tu}{\textup}
\newcommand{\lex}{{\text{lex}}}
\newcommand{\eqdef}{=_{\mathrm{def}}}
\renewcommand{\qedsymbol}{$\Box$}

\newcommand{\pvec}{\vec{p}}
\newcommand{\successor}{\mathrm{succ}}
\newcommand{\ph}{\mathfrak{P}}
\newcommand{\Lp}{\mathrm{Lp}}
\newcommand{\betavec}{\vec{\beta}}
\newcommand{\esomone}{\mathbbm{e}}\newcommand{\Momone}{\mathbbm{m}}
\newcommand{\xibar}{\bar{\xi}}
\newcommand{\lgcd}{\mathrm{lgcd}}
\newcommand{\css}{\mathrm{css}}
\newcommand{\stack}{\mathrm{stack}}

\title{The definability
of the extender sequence $\mathbbm{E}$\\ from $\mathbbm{E}\upharpoonright\aleph_1$ in $L[\mathbbm{E}]$\footnote{This is the author accepted version of an article published in The Journal of Symbolic Logic, 89(2):427--459, 2024, available at \url{https://www.doi.org/10.1017/jsl.2024.27}.}}
\author{Farmer Schlutzenberg\\
TU Wien}

\begin{document}

\maketitle

\begin{abstract}
Let $M$ be a short extender mouse.

We prove that if $E\in M$ and $M\sats$``$E$ is a countably complete short extender
whose support is a cardinal $\theta$ and $\her_\theta\sub\Ult(V,E)$'',
then $E$ is in the extender sequence $\es^M$ of $M$.
We also prove other related facts,
and use them to establish that if $\kappa$ is an uncountable cardinal 
of $M$ and $\kappa^{+M}$ exists in $M$ then $(\her_{\kappa^+})^M$ satisfies the Axiom of Global Choice.

We prove that if $M$ satisfies the Power Set Axiom
then $\es^M$ is definable over the universe 
of $M$ from the parameter 
$X=\es^M\rest\aleph_1^M$, and  $M$ satisfies
``Every set is $\OD_{\{X\}}$''. We also prove various
local versions of this fact in which $M$ has a largest cardinal,
and a version for generic extensions of $M$.

As a consequence, for example, the minimal proper class mouse with 
a Woodin limit of Woodin cardinals models ``$V=\HOD$''.
This adapts to many other similar examples.

We also describe a simplified approach to Mitchell-Steel fine structure, which does away with the parameters $u_n$.\footnote[0]{

MSC2020 class: 03E45, 03E55}
\end{abstract}

\section{Introduction}\label{sec:introduction}

Let $M$ be a premouse.\footnote{All premice considered in the paper will have only short extenders on their extender sequence,
even when not mentioned explicitly.} Write $\es^M$ for the extender sequence of $M$, not
including the active extender $F^M$ of $M$. Write $\es_+^M=\es^M\conc\left<F^M\right>$. Write $\univ{M}$ for the universe of $M$.
Write $\Momone^M=M|\aleph_1^M$. Write $\PS$ for the Power Set Axiom.\label{page:E^M}

We  consider here the following questions:
\begin{enumerate}[label=--]
 \item Given $E\in M$ such that $M\sats$``$E$ is an extender'', is $E\in\es^M$?
  \item (Steel) Suppose $M\sats\ZFC$. Does $M\sats$``There is $X\sub\aleph_1$
such that $V=\HOD_{\{X\}}$''?
 \item Is $\es^M$ definable over $\univ{M}$,
 possibly from some (small) parameter?
\end{enumerate}

Throughout the paper, when we write
``$\es^M$ is definable (within some complexity class)'',
we literally mean that the class
$\{\es^M\rest\alpha\mid\alpha<\OR^M\}$
is so definable. Given a definability class $\Sigma$, We write $\Sigma^N(X)$
for the class of relations which are $\Sigma$-definable over the structure $N$
from parameters in $X$, or just $\Sigma^N$ for $\Sigma^N(\emptyset)$.
As usual, $\Delta_n^N(X)$ denotes $\Sigma_n^N(X)\cap\Pi^N_n(X)$.

Answers to certain instances of the above questions have been known for some time.
Kunen proved that $L[U]$ satisfies ``$U$ is the unique normal measure'',
and therefore satisfies ``$V=\HOD$''.
Recall that $M_n$ is the minimal proper class mouse with $n$ Woodin cardinals.
Steel proved \cite{cmpwc} that for $n\leq\om$, $\es^{M_n}$ is definable over 
$\univ{M_n}$ without 
parameters.\footnote{He in fact showed that $M_n$
is its own ``core model'' (this must be defined appropriately).} The author proved similar results for larger, sufficiently self-iterable mice in 
\cite{mim} and \cite{extmax}. The proofs of these earlier
results depended on the mice in question being sufficiently self-iterable.
But 
non-meek mice typically fail to have such self-iterability,
making it difficult to generalize these kinds of arguments to models with higher large cardinals.

The main result of the paper is the following,
Theorem \ref{thm:E_def_from_e}. It answers Steel's question above positively,
in fact with $X=\Momone^M$.\footnote{
A natural variant of Steel's question 
is whether the same holds for some $X\in\RR^M$.
We do not know the answer of this question in general, but
in \cite{odle},
we will extend the results and methods here
to answer it affirmatively for tame mice.}

\begin{tm}\label{thm:E_def_from_e}
Let $M$ be a $(0,\om_1+1)$-iterable premouse satisfying $\PS$ and $\Momone=\Momone^M$.
Then
\[ \es^M\text{ is }\Delta_2^{\univ{M}}(\{\Momone\})\text{-definable}.\]
Therefore if $\univ{M}\sats\ZFC$ then $\univ{M}\sats$``$V=\HOD_{\{\Momone\}}$'' and $M\sats\ZFC$.\footnote{Note that
by writing ``$M\sats\ZFC$'', we mean ZFC including the Separation and Collection schemata in the premouse language, and refer to the structure $(\univ{M},\in,\es^M)$,
and so the assumption that $\univ{M}\sats\ZFC$ does not trivially imply $M\sats\ZFC$.}
\end{tm}

Note that the only large cardinal restriction on the mice involved is the paper's global assumption
that  all premice considered have only short extenders on their sequence. This means that
we need to deal with mice which are significantly non-self-iterable, for which the earlier arguments mentioned above do not seem to apply. 

For the proof, we will
use a method which avoids any self-iterability,
and is more focused on condensation properties.
The first proof we give, in \S\ref{sec:con_stack},
will actually
yield a more general and local version,
in which the mouse can have a largest cardinal, but in which
case we must allow somewhat higher complexity in the definition of $\es^M$
from the parameter $\Momone^M$. We will also give a variant
proof in \S\ref{sec:direct_con_stack},
which uses the same main idea, but is a little simpler. (However, it does not give all of the information provided by the first proof.)
The argument
in \S\ref{sec:con_stack} was discovered in 2015, and 
that in \S\ref{sec:direct_con_stack} in 2019.
The argument in \S\ref{sec:direct_con_stack} is used in the preprint \cite{vm2}.

We easily get the following corollary,
which does involve some self-iterability:

\begin{cor*}[\ref{cor:V=HOD}]
 Let $M$ be a $(0,\om_1+1)$-iterable premouse.
 Suppose $M$ satisfies $\PS+$``$\Momone^M$ is $(\om,\om_1+1)$-iterable''.
 Then $\es^M$ is \tu{(}lightface\tu{)} $\Delta_2^{\univ{M}}$-definable and
  $\univ{M}\sats$``Every set is $\OD$''.\footnote{Let $M$ be a premouse modelling $\PS$.
 Here and elsewhere, we say that $x\in\OD^{\univ{M}}$ iff
 there is $\alpha<\OR^M$ such that $\{x\}$
  is definable from ordinal parameters over $\her_\alpha^M$.
  Likewise for $\OD_P^{\univ{M}}$.}
\end{cor*}

To prove the more local version of Theorem \ref{thm:E_def_from_e} above
we will use certain \emph{extender maximality} properties of $\es^M$,
Theorems \ref{thm:strong_extender_in_sequence}
and \ref{thm:proj_to_strength} below, which are refinements of results from \cite{mim}, with similar proofs.

\begin{tm}[Steel, Schlutzenberg]\label{thm:strong_extender_in_sequence}
Let $M$ be an $(0,\om_1+1)$-iterable premouse. Let $E\in M$  
be such that $M\sats$``$E$ is a short, total, countably complete extender'',
$\nu_E$ is a cardinal of $M$ and
$\her_{\nu_E}^M\sub\Ult(M,E)$.
Then \tu{(}the trivial completion of\tu{)} $E$ is in $\es^M$.
\end{tm}

\begin{rem}
 Steel first proved that $E\in\es^M$ under the assumptions of \ref{thm:strong_extender_in_sequence}
 together with the added assumptions
 that $M\sats\PS$ + ``$\nu_E$ is regular'' and 
 $\Ult(M,E)|\nu_E=M|\nu_E$.
 The author then generalized Steel's proof to obtain 
\ref{thm:strong_extender_in_sequence}.

Recall that if $M\sats$ ``$E$ is a normal measure'' then $\nu_E=\crit(E)^{+M}$, so the 
requirement that $\her_{\nu_E}^M\sub\Ult(M,E)$ holds automatically,
and therefore $E\in\es^M$ (given $M$ is iterable).
\end{rem}

\begin{tm}\label{thm:proj_to_strength}
Let $M$ be a $(0,\om_1+1)$-iterable premouse. Let $E,R\in M$ and $\tau\in\OR^M$ be such that 
$\tau$ is a cardinal of $M$, $R$ is a premouse, $\rho_\om^R=\tau$ and $M\sats$ ``$E$ is a short 
extender, 
$\her_\tau\sub\Ult(M,E)$ and $R\pins\Ult(M,E)$''.
Then $R\pins M$.
\end{tm}

Slightly less general versions of \ref{thm:strong_extender_in_sequence} and \ref{thm:proj_to_strength}
were obtained by the author in 2006 \cite{mim}. Prior to this, Woodin had conjectured
that if $M$ is a mouse, $\kappa$ is uncountable in $M$ and $\kappa^{+M}<\OR^M$,
then $L(\pow(\kappa)^M)\sats\AC$.
Woodin's conjecture follows immediately from the following corollary (\ref{cor:L(pow(kappa))^M_AC})
to the preceding theorems.
Steel noticed that \ref{cor:L(pow(kappa))^M_AC} follows from
\ref{thm:proj_to_strength} combined with an argument of Woodin's.\footnote{
The corollary appeared first in \cite{mim}. It can now be deduced trivially
from \ref{thm:E_def_from_e}. However, we will give its original proof (from \cite{mim}),
as this constitutes a significant part of the proof of \ref{thm:E_def_from_e},
and so it serves as a useful warm-up.}

\begin{cor}\label{cor:L(pow(kappa))^M_AC} Let $M$ be a $(0,\om_1+1)$-iterable
premouse and $\kappa\in\OR^M$ be such that $M\sats$``$\kappa$ is uncountable''
and $\kappa^{+M}<\OR^M$.
Then $M|\kappa^{+M}$ is definable from parameters over $\her_{\kappa^{+M}}^M$.
\end{cor}

In \S\ref{sec:fs} we also describe a simplification to Mitchell-Steel fine structure of \cite{fsit},
making do without the parameters $u_n$.
One could just use the standard fine structure,
but the simplification removes some  complications, and
we will officially make use of it throughout the paper.

\subsection{Conventions and notation}\label{ssec:ntn}

Note that some notation was introduced on page \pageref{page:E^M}.

Most non-standard conventions
are as in \cite[\S1.1]{premouse_inheriting} or \cite[\S1.1]{iter_for_stacks} (and mostly, but not completely,
as in \cite[\S1.1]{extmax}).
In particular,
 \emph{premice}  $M$
have Mitchell-Steel indexing,
except that we allow
extenders of superstrong type in the extender sequence $\es_+^M$ (see \cite[\S2]{operator_mice},
 \cite[\S\S1.1.2, 1.1.6]{premouse_inheriting}),
 and we adopt a simplified version of the Mitchell-Steel fine structure of \cite{fsit},
 which avoids the parameters
$u_n$, as explained in \S\ref{sec:fs}.
By \ref{thm:fs}, this change actually has no impact
on the fine structural notions
such as standard parameters, $k$-soundness, etc. In \S\S\ref{sec:introduction}--\ref{sec:direct_con_stack}, when we write $p_k$ (the $k$th standard parameter), we mean the object defined as $q_k$ in \S\ref{sec:fs}, as opposed to what is defined as $p_k$ in \S\ref{sec:fs}. Moreover, we drop the ``q-'' from the fine structure terminology  introduced in \S\ref{sec:fs}.
Because of the change, we use the notation $\Hull_{k+1}^M(X)$ and $\cHull_{k+1}^M(X)$
as defined in \ref{dfn:Hull_k+1},
not as in \cite{extmax}.

For a structure $M$, $\univ{M}$ denotes the universe of $M$, and $\J(M)$ denotes the rudimentary closure of $M\cup\{M\}$. 

Let $N$ be a premouse, $\es=\es^N$ and $\es_+=\es_+^N$.
 Given $\alpha\leq\OR^N$,
we write $N|\alpha=(\J_\alpha^{\es},\es\rest\alpha,\es_+(\alpha))$ for the initial segment of $N$ of ordinal height $\alpha$, including its active extender $\es_+(\alpha)$,
and $N||\alpha$ for its passivization $(\J_\alpha^{\es},\es\rest\alpha,\emptyset)$.
If $N$ is passive, then working inside $N$, $\J^\es$ also denotes $N$. We write $\esomone^N=\es\rest\om_1^N$ and $\Momone^N=N|\om_1^N$.

 Let $n<\om$. We say that $N$ satisfies \emph{$(n+1)$-condensation}
 iff $N$ is $n$-sound and whenever $H$ is $(n+1)$-sound and $\pi:H\to N$ is
 $n$-lifting (see \cite[Definition 2.1]{premouse_inheriting})
 and $\rho_{n+1}^H\leq\crit(\pi)$, then either $H=\core_{n+1}(N)$ or $H\pins N$
 or, letting $\rho=\rho_{n+1}^H$, $N|\rho$ is active with extender $E$ and 
$H\pins\Ult(N|\rho,E)$. (See \cite[Theorem 5.2]{premouse_inheriting}.) We say $N$ satisfies \emph{$\om$-condensation}
iff it satisfies $(n+1)$-condensation for all $n<\om$.

Regarding (generalized) solidity, see Definition \ref{dfn:solidity} and \cite[\S1.1.3]{premouse_inheriting}.

We say that $N$ is an \emph{$\om$-premouse} iff $N$ is $\om$-sound and 
$\rho_\om^N=\om$; in this case we let $\deg(N)$ denote the least $n$ such that 
$\rho_{n+1}^N=\om$. An \emph{$\om$-mouse} is an $(\om,\om_1+1)$-iterable $\om$-premouse.
If $N$ is an $\om$-mouse, we write $\Sigma_N$ for the unique $(\om,\om_1+1)$-strategy for $N$.

For $\alpha<\OR^N$, recall that $\alpha$ is a \emph{cutpoint} of $N$ iff 
for all $E\in\es_+^N$,
if $\crit(E)<\alpha$ then $\lh(E)\leq\alpha$.

For an extender $E$, $t_E$ and $\tau_E$ denote the Dodd parameter
and Dodd projectum of $E$ respectively, if they are defined.

\section{Extender maximality}\label{sec:extender_maximality}

In this section we prove Theorems \ref{thm:strong_extender_in_sequence}
and \ref{thm:proj_to_strength}. The proofs are refinements of less general results proved in 
\cite{mim}.
Toward these proofs, we begin with a lemma which helps us to find sound hulls of 
premice; the proof is basically as in \cite[Lemma 3.1]{extmax},
but here we use the fact that condensation
follows from normal iterability in order to reduce our assumptions.

\begin{dfn}\label{dfn:solidity}
 Let $k<\om$, let $H$ be $k$-sound, $q\in[\rho_0^H]^{<\om}$ and $\alpha\in\OR^H$.
 The \emph{$(k+1)$-solidity witness for $(H,q,\alpha)$},
 (or just \emph{for $(q,\alpha)$}), is
 \[ w^H_{k+1}(q,\alpha)\eqdef\Th_{k+1}^H(\alpha\cup\{q,\pvec_k^H\}). \]
 Letting $q=\{q_0,\ldots,q_{\lh(q)-1}\}$ with $q_i>q_{i+1}$,
 the (set of all) \emph{$(k+1)$-solidity witnesses for $(H,q)$}
 (or just \emph{for $q$})
is
\[ w^H_{k+1}(q)\eqdef\{w_{k+1}^H(q\rest i,q_i)\}_{i<\lh(q)}\]
where $q\rest i=\{q_0,\ldots,q_{i-1}\}$.
The (set of all) \emph{$(k+1)$-solidity witnesses for $H$}
is
\[ w^H_{k+1}\eqdef w^H_{k+1}(p_{k+1}^H).\qedhere\]
\end{dfn}

Note that in the preceding definition, we are not assuming that the solidity witnesses
in consideration are in $H$.

\begin{dfn}
Let $k<\om$, let $H$ be $(k+1)$-sound, $q\in\core_0(H)$, $\theta<\rho_0^H$,
\[ \Hbar=\cHull_{k+1}^H(\theta\un\{\pvec_k^H,q\}), \]
$\pi:\Hbar\to H$ be the uncollapse and $\pi(\qbar)=q$. We say that $(\theta,q)$ is 
\emph{$(k+1)$-self-solid \tu{(}for $H$\tu{)}} iff $\Hbar$ is $k+1$-sound and 
$\rho_{k+1}^\Hbar=\theta$ and 
$p_{k+1}^\Hbar=\bar{q}$.

Let $x\in\core_0(H)$ and $r\in[\rho_0^H]^{<\om}$. We say that $r$ is an $\rSigma_{k+1}^H(\{x\})$-generator 
iff for every $\gamma\in r$, we have
\[ \gamma\notin\Hull_{k+1}^H(\gamma\un\{\pvec_k^H,x,r\backslash\{\gamma\}\}).\qedhere\]
\end{dfn}

\begin{lem}\label{lem:sound_hull}
 Let $k<\om$ and let $H$ be $(k+1)$-sound and $(k,\om_1+1)$-iterable.
 Let $r\in \core_0(H)$ and 
$\theta\leq\rho_{k+1}^H$ be a cardinal of $H$. Then there is $q\in H$ such 
that:
\begin{enumerate}[label=--]
 \item $(\theta,q)$ is $(k+1)$-self-solid for $H$,
 \item $p_{k+1}^H=q\backslash\min(p_{k+1}^H)$,
 \item $r\in\Hull_{k+1}^H(\theta\un\{\pvec_k^H,q\})$, and
 \item $\Hbar=\cHull_{k+1}^H(\theta\un\{\pvec_k^H,q\})\ins H$.
\end{enumerate}
\end{lem}
\begin{proof}
We may assume $H$ is countable and $\theta<\rho_{k+1}^H$. We will define $m<\om$ and
\[ q=(q_0,q_1,\ldots,q_{m-1}),\] 
with $q_i>q_{i+1}$ for $i+1<m$. Let $p=p_{k+1}^H$. We start with $q\rest\lh(p)=p$. We define 
$q_i$ for $i\geq\lh(p)$ by induction on $i$, with $q_i<\rho_{k+1}^H$.
We simultaneously define an $H$-cardinal
$\gamma_i$, with $\gamma_{\lh(p)}=\rho_{k+1}^H$
and $\theta\leq\gamma_{i}\leq q_{i-1}$ for $i>\lh(p)$, and
\[ u_i,r\in\Hull_{k+1}^H(\gamma_i\un\{\pvec_k^H,q\rest i\}), \]
where $u_i$ is the set of $(k+1)$-solidity witnesses for $(H,q\rest i)$. Now let $i\geq\lh(p)$, and let 
$q\rest i$, $\gamma_i$ be given. If $\gamma_i=\theta$ then we set $m=i$, so 
$q=q\rest i$ and we are done. So suppose $\gamma_i>\theta$. Let $\eta<\gamma_i$ be least such that
$\eta>\theta$ and $\eta$ is not a cardinal of $H$ and
\begin{equation}\label{eqn:def_H_i} u_i,r\in H_i\eqdef\Hull_{k+1}^H((\eta+1)\un\{\pvec_k^H,q\rest i\}) \end{equation}
and
\begin{equation}\label{eqn:eta_generator} \eta\notin\Hull_{k+1}^H(\eta\un\{\pvec_k^H,q\rest i\}). 
\end{equation}
Let $q_i=\min(\OR\backslash H_i)=\gamma_i^{+H}\cap H_i$ and let $\gamma_i=\card^H(q_i)=\card^H(\eta)$.

Clearly
\[ H_i\psub H_i'=\Hull_{k+1}^H(\gamma_i\un\{\pvec_k^H,q\rest(i+1)\}), \]
so it suffices to see that $u_{i+1}\in H_i'$. Note that the transitive collapse $W_i$ of $H_i$ 
is (equivalent to) the $(k+1)$-solidity witness for $(q\rest i,q_i)$, so it suffices to see that $W_i\in H_i'$.
For this, noting that $q_i=\gamma_i^{+W_i}$, it suffices to see that $W_i\pins H$, since then $W_i$ is the least segment $W$ of $H$ 
such that $\OR^W\geq q_i$ and $\rho_\om^W=\gamma_i=\lgcd(H|q_i)$.

Let $\rho=q_i$ and $\gamma=\gamma_i$ and $W=W_i$. Let $\pi:W\to H$ be the uncollapse. Then 
$\pi(p_{k+1}^W\backslash\rho)=q\rest i$ and $W$ is $\rho$-sound and $\crit(\pi)=\rho$ and 
\[ \rho>\rho_{k+1}^{W}=\gamma=\lgcd(W|\rho) \]
($\rho_{k+1}^{W}\geq\gamma$ because $W\in 
H$, and $\rho_{k+1}^W<\rho$ because $\eta<\rho$
and by line (\ref{eqn:def_H_i})). So by condensation as stated in \cite[Theorem 5.2]{premouse_inheriting}, either (a) $W\pins H$ or (b)
letting $J\pins H$ be least such that $q_i\leq\OR^J$ and $\rho_\om^J=\gamma$,
then $\rho_{k+1}^J=\gamma<\rho_k^J$ and there is a type 1 extender $F$ over $J$
with $\crit(F)=\gamma$ and
$W=\Ult_k(J,F)$.
But since $\eta>\gamma$ and 
because of line (\ref{eqn:eta_generator}), we have
\[ \eta\notin\Hull_{k+1}^W((\gamma+1)\un\{\pvec_k^W,p_{k+1}^W\backslash\rho\}), \]
and therefore (b) is false. So $W\pins H$, as required.

Since $\gamma_{i+1}<\gamma_i$, the construction terminates successfully.

Finally, the fact that $\Hbar\ins H$ (where $\Hbar$ is defined in the statement of the theorem) 
follows from condensation.
\end{proof}
Related calcuations also give the following:
\begin{lem}\label{lem:hulls_proper_segs}
 Let $k<\om$ and let $H$ be $(k+1)$-sound and $(k,\om_1+1)$-iterable. Suppose $\rho=\rho_{k+1}^H=\kappa^{+H}>\om$ and $\kappa$ is an $H$-cardinal.
For $\gamma<\rho$ let
\[ H_\gamma=\Hull_{k+1}^H(\gamma\un\{\pvec_{k+1}^H\}) \]
and $W_\gamma$ be the transitive collapse of $H_\gamma$.
Then:
\begin{enumerate}[label=\tu{(}\roman*\tu{)}]
 \item\label{item:ev_prop_seg_M_or_ult} 
For all sufficiently large $\gamma\in(\kappa,\rho)$, 
either:
\begin{enumerate}[label=--]
\item $W_\gamma\pins H$, or
\item $\kappa^{+W_\gamma}=\rho_\om^{W_\gamma}=\gamma$, $H|\gamma$ is active with $E$\footnote{Note then that $\crit(E)<\kappa$
 and $E$ is $H$-total.} and
$W_\gamma\pins\Ult(H,E)$.
\end{enumerate}
\item\label{item:cof_prop_seg_M}For cofinally many $\gamma<\rho$, we have $W_\gamma\pins H$
 and $\rho_{k+1}^{W_\gamma}=\kappa$.
 \end{enumerate}
\end{lem}
\begin{proof}
For $\gamma\in(\kappa,\rho)$, say that $\gamma$ is a \emph{generator} iff $\gamma\notin H_\gamma$.
We say that a generator is a \emph{limit generator} iff it is a limit of generators,
and is otherwise a \emph{successor generator}.
Note that the set of generators above $\kappa$ is club in $\rho$. 
Let $\gamma>\kappa$ be a generator. Then note that $W_\gamma\in H$ and
$\gamma=\kappa^{+W_\gamma}$ so $\rho_{k+1}^{W_\gamma}\in\{\kappa,\gamma\}$; moreover, if $\gamma$ is a successor generator
then $\rho_{k+1}^{W_\gamma}=\kappa$. 

Now let $\eta_0$ be the least generator $\gamma>\kappa$ such that
 $w_{k+1}^H\in H_\gamma$.
We claim that the conclusion of \ref{item:ev_prop_seg_M_or_ult} holds for all generators $\gamma>\eta_0$. We proceed by 
induction on $\gamma$.

First suppose that $\gamma$ is a limit generator. Then by induction,
for eventually all successor generators $\gamma'<\gamma$, we have
$W_{\gamma'}\pins H$ and $\gamma'=\kappa^{+W_{\gamma'}}$ and
$W_{\gamma'}$ projects to $\kappa$. It follows that $W_{\gamma'}\in H_\gamma$,
so $W_{\gamma'}\in W_\gamma$, which implies that $\rho_{k+1}^{W_{\gamma}}=\gamma$,
and therefore $W_\gamma$ is $(k+1)$-sound. So the conclusion for $W_\gamma$ follows 
from $(k+1)$-condensation.

Now suppose that $\gamma$ is a successor generator. Then there is a largest generator 
$\eta<\gamma$, and we have $\kappa<\eta_0\leq\eta<\gamma$,
and $W_\gamma$ projects to $\kappa$.
So using
condensation (as stated in \cite[Theorem 5.2]{premouse_inheriting}) as in the proof of \ref{lem:sound_hull},
we get $W_\gamma\pins H$. 

Part \ref{item:cof_prop_seg_M} now easily follows;
in fact its conclusion holds for every sufficiently large successor generator.
\end{proof}

\begin{rem}
 Let $M$ be an $m$-sound premouse.
 Recall that a (putative) iteration tree on $M$ is \emph{$m$-maximal}
 given that (i) $\Tt$ satisfies the \emph{monotone length condition}
 \[ \lh(E^\Tt_\alpha)\leq\lh(E^\Tt_\beta)\text{ for all }\alpha+1<\beta+1<\lh(\Tt),\]
 and for each $\alpha+1<\lh(\Tt)$, (ii) $\gamma=\pred^\Tt(\alpha+1)$
 is least such that $\crit(E^\Tt_\alpha)<\nu(E^\Tt_\gamma)$,
 (iii) 
 $M^{*\Tt}_{\alpha+1}\ins M^\Tt_\gamma$ is as large as possible, and 
 (iv) $k=\deg^\Tt_{\alpha+1}$ is as large as possible subject to the choice of $M^{*\Tt}_{\alpha+1}$ (with $k\leq\deg^\Tt_\gamma$
 if $M^{*\Tt}_{\alpha+1}=M^\Tt_\gamma$).
\end{rem}

\begin{dfn}Let $M$ be an $m$-sound premouse.
An \emph{essentially $m$-maximal tree} on $M$
satisfies the requirements of $m$-maximality,
except that we drop the monotone length condition,
replacing it with the \emph{montone $\nu$ condition}, that is, that
\[ \nu(E^\Tt_\alpha)\leq\nu(E^\Tt_\beta)\text{ for all }\alpha+1<\beta+1<\lh(\Tt).\qedhere\]
\end{dfn}

\begin{rem}\label{rem:essential_m-max_iter}
It is easy to see that, for example, $(m,\om_1+1)$-iterability
is equivalent to essential-$(m,\om_1+1)$-iterability.
\end{rem}

\begin{dfn}
 Let $\pi:\core_0(M)\to\core_0(N)$ be
 $\Sigma_0$-elementary between premice $M,N$ of the same type.
 
 If $M,N$ are passive then $\psi_\pi$ denotes $\pi$.
 If $M,N$ are active, $\mu=\crit(F^M)$ and $\kappa=\crit(F^N)$, then
 \[ \psi_\pi:\Ult(M|\mu^{+M},F^M)\to\Ult(N|\kappa^{+N},F^N) \]
 denotes the embedding induced by the Shift Lemma from $\pi$.
 So in both cases, $\pi\sub\psi_\pi$ and $\psi_\pi$ is fully elementary.
 
 Now we say that $\pi$ is:
 \begin{enumerate}[label=--]
  \item \emph{$\nu$-low} iff $M,N$ are type 3 and $\psi_\pi(\nu^M)<\nu^N$,
  \item \emph{$\nu$-preserving} iff [if $M,N$ are type 3 then $\psi_\pi(\nu^M)=\nu^N$], and
  \item \emph{$\nu$-high} iff $M,N$ are type 3 and $\psi_\pi(\nu^M)>\nu^N$.\qedhere
 \end{enumerate}
\end{dfn}

\begin{rem}
Suppose $\pi,M,N$ are as above and $M,N$ are type 3.
It is easy to see that if $\pi$ is $\rSigma_2$-elementary then $\pi$ is $\nu$-preserving,
and if $\pi$ is $\rSigma_1$-elementary then $\pi$ is non-$\nu$-low.
Moreover, one can show that if $\pi=\pi_0$ is a $\nu$-preserving
near $k$-embedding, then the copying construction for $k$-maximal (or essentially $k$-maximal) trees with $\pi$ preserves
tree order, and for each $\alpha$, $\pi_\alpha$ is a $\nu$-preserving
near $\deg^\Tt_\alpha$-embedding. (Here if $M^\Tt_\alpha$
is type 3 and $\rho_0(M^\Tt_\alpha)<\lh(E^\Tt_\alpha)<\OR(M^\Tt_\alpha)$
then we copy $E^\Tt_\alpha$ to $E^\Uu_\alpha=\psi_{\pi_\alpha}(E^\Tt_\alpha)$.)
\end{rem}

\begin{dfn}
 Let $M$ be an active premouse, $F=F^M$ and $\kappa=\crit(F)$.
 We say $F$ is of \emph{superstrong type}
 iff $i^M_F(\kappa)$ is the largest cardinal of $M$.
 We say a premouse $N$ is \emph{below superstrong}
 iff no $E\in\es_+^N$ is of superstrong type.
\end{dfn}

We will primarily be interested in extenders whose support
is of form $\sigma\cup t$ where $\sigma\in\OR$ and $t$ is a finite set of ordinals.  However, we do not want to assume explicitly that the ``ordinals'' in $t$ are in fact wellfounded, and we want to allow that there are other ``ordinals''
of the ultrapower below ``$\max(t)$'' which are not themselves in the support. (This can be the case, for example, for an extender of form $E\rest(\tau_E\cup t_E)$.)
Toward this we adopt the following terminology:

\begin{dfn}\label{dfn:standard_ext}
 Say a short extender $E$ over $\kappa=\crit(E)$ is \emph{standard}
 if it has support $\sigma+n$
 for some limit ordinal
 $\sigma$ and $n<\om$,
 (so $E=\left<E_a\right>_{a\in[\sigma+n]^{<\om}}$), $[\kappa]^{|a|}\in E_a$ for each $a\in[\sigma+n]^{<\om}$,
 and letting $\id':[\kappa]^1\to\kappa$ be the function $\id'(\{\alpha\})=\alpha$, for each $\alpha<\sigma$, we have
 $[\{\alpha\},\id']_E=\alpha$. (We follow the usual conventions regarding how $E_{a\cup b}$ projects to $E_a$ when $a,b\in[\sigma+n]^{<\om}$.
 So for each $\alpha<\sigma$ and $m_0\leq m_1<n$,
 $[\{\sigma+m_i\},\id']_E$ represents an ``ordinal'' $\beta_i$ of the ultrapower and the ultrapower satisfies ``$\alpha<\beta_0<\beta_1$'', though we do not assume that $\beta_0,\beta_1$ are in the wellfounded part of the ultrapower.)
\end{dfn}

 In what follows, in the above context, when $n>0$, we typically write $t$ for $\{\sigma,\ldots,\sigma+n-1\}$,
 but may also identify $t$ with the finite set $\{[\{\sigma+m\},\id']_E\bigm|m<n\}$ of ordinals of the ultrapower.

We will deduce Theorems \ref{thm:strong_extender_in_sequence} and \ref{thm:proj_to_strength} from the following:

\begin{tm}\label{tm:ext_tau_a_card}
Let $N$ be a $(0,\om_1+1)$-iterable premouse, $F\in N$ and $\mu,\sigma\in\OR^N$, $\ell<\om$
 and $W$ be such that:
 \begin{enumerate}[label=--]
 \item $\sigma$ is an $N$-cardinal,
 \item $F$ is a standard short $N$-extender
 with support $\sigma+\ell$,
 weakly amenable to $N$, coded as a subset of $N|\sigma$,
 such that $N\sats$``$F$ is countably complete'',
 \item $W=\Ult_0(N,F)$, $\mu=\crit(F)<\sigma$ and $\her_\sigma^N\sub W$.
\end{enumerate} 
 Then \tu{(}i\tu{)} $W|\sigma^{+W}=N||\sigma^{+W}$ and if $\ell=0$ then \tu{(}ii\tu{)} $F\in\es_+^{N}$.
\end{tm}
\begin{proof}
We may assume that $N=\J(R)$ where $F$ is definable from parameters over $R$
and $\rho_\om^R=\sigma$. Say that $F$ is $\rSigma_{n+1}^R(\{r\})$.
We may also assume inductively that all segments of $R$
satisfy the theorem. 

Let $n\ll m<\om$ and
$M=\cHull_{m+1}^R(\{s\})$
where $(\om,s)$ is $(m+1)$-self-solid for $R$
and $r\in\rg(\pi_{MR})$ where $\pi_{MR}:M\to R$ is the uncollapse.

Let $E=(\pi_{MR}^{-1})``F$. So $E$ is a
standard $M$-extender with support $\tau+\ell$,
where  either $\tau<\rho_0^M$ and $\pi_{MR}(\tau)=\sigma$,
 or $\tau=\rho_0^M$ and $\rho_0^R=\sigma$. And
$E$ is $\bfrSigma_{n+1}^M$-definable,
$M$ is $(m+1)$-sound, $n+10< m$ and
\[ \rho_{m+1}^M=\om<\kappa=\crit(E)<\tau=\rho_{m}^M=\rho_{n+1}^M. \]
Other relevant properties of $(R,F)$ also reflect to $(M,E)$.
Moreover,
\begin{equation}\label{eqn:U_iter} U=\Ult_m(M,E)\text{ is wellfounded and }(m,\om_1+1)\text{-iterable},\end{equation}
by the countable completeness of $F$ in $N$ and the $(\om,\om_1+1)$-iterability of $R$.

Now $\tau<\OR^M$. For suppose $\tau=\OR^M$.
Since $\rho_{n+10}^M=\tau$, therefore $M$ is passive.
If $\tau=\kappa^{+M}$ (that is, $\kappa$ is the largest cardinal of $M$),
then we have $U|\kappa^{+M}=M|\kappa^{+M}=M$ (by condensation for $M$),
but then $E\in U$, which is impossible. So $\tau>\kappa^{+M}$.
Then $E\rest\eta\in M$ for all $\eta<\tau$ (since $\rho_{n+10}^M=\tau$),
so by induction (with conclusion (i) of the theorem), $U|\tau=M|\tau=M$, so again, $E\in U$, a contradiction.

Let $\theta$ be the largest $M$-cardinal $\leq\tau$ such that
$M|\theta=U|\theta$.\footnote{We will show that $\theta=\tau$.}
Let $t$ be $(m,\om)$-self-solid for $M$,
and such that letting
\[ \bar{M}=\cHull_m^M(\{t\}) \]
and $\pi:\bar{M}\to M$ be the uncollapse, then $r,s\in\rg(\pi_{MR}\com\pi)$,
and $\tau\in\rg(\pi)$ if $\tau<\rho_0^M$, and $\theta\in\rg(\pi)$ if $\theta<\rho_0^M$. Let $\pi(\bar{t})=t$, etc, and $\bar{E}=(\pi^{-1})``E$, etc.
So $\bar{E}$ is defined over $\bar{M}$ from $\bar{t}$ just as $E$ is over $M$ from $t$,
and the relevant properties of $(M,E,U)$ reflect to $(\bar{M},\bar{E},\bar{U})$, where
$\bar{U}=\Ult_{m-1}(\bar{M},\bar{E})$.

Let $\pi(\bar{\theta})=\theta$ if $\theta<\rho_0^M$,
and otherwise $\bar{\theta}=\rho_0^{\bar{M}}$.
Likewise for $\bar{\tau}$.
So $\bar{\theta},\bar{\tau}$ have the same defining properties with respect to $\bar{M},\bar{U}$.
Define the phalanx (see \cite[\S1.1]{extmax} for the notation)
\[\ph=((\bar{M},m-1,\bar{\theta}),(\bar{U},m-1),\bar{\theta}).\]

\begin{clm} $\ph$ is $(\om_1+1)$-iterable. 
\end{clm}
\begin{proof}
We will lift trees on $\ph$ to essentially $m$-maximal trees on $U$,
which by \ref{rem:essential_m-max_iter} and line (\ref{eqn:U_iter}) suffices. 
Let $\psi:\bar{U}\to U$ be the Shift Lemma map.
Let $\theta'=\sup\pi``\bar{\theta}$.

\begin{case}
 $\theta'<\theta$.

Let $\gamma=\card^{M}(\theta')$, so $\gamma<\theta$.
Let $t'$ be 
such that $(\gamma,t')$ is $m$-self-solid for $M$,
with
\[ M'=\cHull^{M}_m(\gamma\cup\{t'\})\pins M, \]
and $\widetilde{\pi}:M'\to M$ be the uncollapse,
such that $t\in\rg(\widetilde{\pi})$.
Let 
$\pi':\bar{M}\to M'$
be $\pi'=\widetilde{\pi}^{-1}\com\pi$.
So
\[ \pi'\rest\bar{\theta}=\pi\rest\bar{\theta}=\psi\rest\bar{\theta}.\]
Note that $\OR^{M'}<\theta$, so $M'\pins U$.

We can use $(\pi',\psi)$ to lift trees $\Tt$ on $\ph$ to essentially $m$-maximal
trees $\Uu$ on $U$.
In case $\theta$ is a limit cardinal of $M$ then everything here is routine (and we actually get $m$-maximal trees
on $U$). So assume that $\theta=\gamma^{+M}$.
Most of the details of the copying process are routine,
but we explain enough that we can point out how the wrinkles are dealt with.
Let $\pi(\bar{\gamma})=\gamma$.
For $\alpha<\lh(\Tt)$ with $\alpha>0$,
let $\mathrm{root}^\Tt_\alpha=0$ if $M^\Tt_\alpha$ is above $\bar{U}$,
and $\mathrm{root}^\Tt_\alpha=-1$ if above $\bar{M}$.
Let $\alpha<\lh(\Tt)$. If $\mathrm{root}^\Tt_\alpha=0$
then the copy map 
\[ \pi_\alpha:M^\Tt_\alpha\to M^\Uu_\alpha \]
is produced routinely. Suppose $\mathrm{root}^\Tt_\alpha=-1$.
If $(-1,\alpha]_\Tt$ does not drop in model and
\[ \crit(i^\Tt_{\beta\alpha})<i^\Tt_{0\beta}(\bar{\gamma})\text{ for all }\beta\in(-1,\alpha)_\Tt \]
(note this does not include $\beta=-1$, so we are allowing $\bar{\gamma}=\crit(i^\Tt_{-1,\alpha})$), then $[0,\alpha]_\Uu$ does not drop in model or degree and
\[ \pi_\alpha:M^\Tt_\alpha\to Q_\alpha=i^{\Uu}_{0\alpha}(M')\pins M^\Uu_\alpha,\]
and $\pi_\alpha$ is an $(m-1)$-lifting embedding which is produced in the obvious manner via the Shift Lemma.
Otherwise, $(0,\alpha]_\Uu$ drops in model, and
$\pi_\alpha:M^\Tt_\alpha\to M^\Uu_\alpha$,
which is again produced in the obvious manner. We copy extenders using these maps.
There is a wrinkle when
$\pred^\Tt(\alpha+1)=-1$ and $\crit(E^\Tt_\alpha)=\bar{\gamma}$,
so consider this case.
We have then $\crit(E^\Uu_\alpha)=\gamma$.
Because
\[ \bar{\gamma}^{+\bar{U}}=\bar{\gamma}^{+\bar{M}}=\bar{\theta}<\lh(E^\Tt_0)\]
and
\[ \gamma^{+U}=\gamma^{+M}=\theta<\lh(E^\Uu_0),\]
we get $M^{*\Tt}_{\alpha+1}=\bar{M}$, and $M^{*\Uu}_{\alpha+1}=U$ (not $M'$),
and $M'\pins U|\theta$. Now if $E^\Uu_\alpha$ is not of superstrong type
then
\[ \lh(E^\Uu_\alpha)<i^\Uu_{0,\alpha+1}(\gamma)<\OR^{Q_{\alpha+1}} \]
and
\[ \pi_{\alpha+1}(\lh(E^\Tt_{\alpha}))=\lh(E^\Uu_{\alpha}) \]
and things are standard. However, if $E^\Uu_\alpha$ is of superstrong type,
then
\[ Q_{\alpha+1}=i^\Uu_{0,\alpha+1}(M')\pins M^\Uu_{\alpha+1}||\lh(E^\Uu_\alpha),\]
so when we lift $E^\Tt_{\alpha+1}$, we get
$\lh(E^\Uu_{\alpha+1})<\lh(E^\Uu_\alpha)$.
However,
\[ \pi_{\alpha+1}(\lambda(E^\Tt_\alpha))=\lambda(E^\Uu_\alpha).\]
Now we claim that $E^\Tt_\alpha$ is also superstrong,
and therefore $\nu(E^\Tt_\alpha)=\lambda(E^\Tt_\alpha)$
and $\nu(E^\Uu_\alpha)=\lambda(E^\Uu_\alpha)$, and then it follows that
\[ \nu(E^\Uu_\alpha)\leq\nu(E^\Uu_{\alpha+1}),\]
as required for the monotone $\nu$-condition.

So suppose $E^\Tt_\alpha$ is not superstrong.
So $\nu(E^\Tt_\alpha)<\lambda(E^\Tt_\alpha)$, so
\[ \pi_{\alpha+1}(\nu(E^\Tt_\alpha))=\psi_{\pi_\alpha}(\nu(E^\Tt_\alpha))<\lambda(E^\Uu_\alpha)=\nu(E^\Uu_\alpha),\]
which implies that $E^\Tt_\alpha=F(M^\Tt_\alpha)$ and
$\pi_\alpha$ is $\nu$-low. In particular, $\pi_\alpha$ is not $\rSigma_1$-elementary,
so is not a near $0$-embedding. Let $j=\mathrm{root}^\Tt_\alpha\in\{-1,0\}$.
By the proof that the copying construction propagates near embeddings
(see \cite{fs_tame}), $(j,\alpha]_\Tt$ does not drop in model,
and so $\bar{M},\bar{U}$ are active. But because $\bar{U}=\Ult_{m-1}(\bar{M},\bar{E})$
and
\[ \crit(\bar{E})\leq\bar{\gamma}<\bar{\theta}\leq\lambda(\bar{E}),\]
we have $\bar{\gamma}\neq\crit(F^{\bar{U}})$,
and then similarly, as $\bar{\theta}\leq\lambda(E^\Tt_0)$,
it easily follows that $j=-1$. But then $\crit(i^\Tt_{j\alpha})\leq\bar{\gamma}$
and $\bar{\theta}\leq\lambda(E^\Tt_0)$, so $\bar{\gamma}\neq\crit(F(M^\Tt_\alpha))$, contradiction.

So $\nu(E^\Tt_\alpha)\leq\nu(E^\Tt_{\alpha+1})$, as desired.
This is the only situation in which the monotone length condition
can fail. We leave the remaining details of the lifting process to the reader.

\end{case}

\begin{case}
 $\pi``\bar{\theta}$ is unbounded in $\theta$.
 
In this case we do not see how to produce a single map lifting $\bar{M}$,
and instead produce a sequence of maps. Note that $\theta$ is a limit cardinal of $M$
(by the case hypothesis we have an $\bfrSigma_m^M$-singularization
of $\theta$, and if $\theta=\gamma^{+M}$ this routinely
implies that $\rho_m^M<\theta$, a contradiction),
and so $\bar{\theta}$ is a limit cardinal of $\bar{M}$.
For each $\bar{M}$-cardinal $\gamma<\bar{\theta}$,
let $(M'_\gamma,\sigma_\gamma)$ be such that
$M'_\gamma\pins M|\theta$ and
$\sigma_\gamma:\bar{M}\to M'_\gamma$
is a near $(m-1)$-embedding with
\[ \sigma_\gamma\rest\gamma^{+\bar{M}}=\pi\rest\gamma^{+\bar{M}}=\psi\rest\gamma^{+\bar{M}}\]
and
$\rho_m^{M'_\gamma}=\sigma_\gamma(\gamma)^{+M}$;
we get such pairs by taking appropriate hulls much as in the previous case.

Now for each $\gamma$ we have $M'_\gamma\pins U$.
So we can use $(\left<\sigma_\gamma\right>_{\gamma<\bar{\theta}},\psi)$
to lift trees on $\ph$ to $m$-maximal trees on $U$.
This is much as in the previous
case, but this time when $\crit(E^\Tt_\alpha)=\gamma<\bar{\theta}$,
then we define $Q_{\alpha+1}=i^\Uu_{0,\alpha+1}(M'_\gamma)$
and define $\pi_{\alpha+1}$ via the Shift Lemma from $\sigma_\gamma$ and $\pi_\alpha$.
We get the monotone length condition here, because
$\sigma_\gamma(\gamma)^{+}<\OR^{M'_\gamma}$.
The details are left to the reader.\qedhere
\end{case}
\end{proof}

Using the claim, we can now complete the proof. We get a successful comparison $(\Tt,\Uu)$
of $(\bar{M},\ph)$, with $\Tt$ being $(m-1)$-maximal. Note that all extenders used in the comparison 
have length $>\bar{\theta}$.
Standard fine structural arguments show that $b^\Uu$ is above $\bar{U}$
and  $b^\Tt,b^\Uu$ are non-model-dropping,
$M^\Tt_\infty=Q=M^\Uu_\infty$
and $\deg^\Tt_\infty=m-1=\deg^\Uu_\infty$.
So $\bar{\theta}\leq\crit(i^\Uu)$, so
$\bar{U}|\bar{\theta}^{+\bar{U}}=Q|\bar{\theta}^{+Q}$,
and since $\lh(E^\Tt_0)>\bar{\theta}$, therefore
\begin{equation}\label{eqn:Mbar_Ubar_strong_agmt} \bar{U}|\bar{\theta}^{+\bar{U}}=\bar{M}||\bar{\theta}^{+\bar{U}}.\end{equation}
But if $\bar{\theta}<\bar{\tau}$ then because $\her_{\bar{\tau}}^{\bar{M}}\sub\bar{U}$,
it follows that
$\bar{U}|\bar{\theta}^{+\bar{U}}=\bar{M}|\bar{\theta}^{+\bar{M}}$,
which contradicts the choice of $\bar{\theta}$. So $\bar{\theta}=\bar{\tau}$,
which with line (\ref{eqn:Mbar_Ubar_strong_agmt}) gives the statement of conclusion (i)
of the theorem but with $\bar{M}$ instead of $N$. However, this statement
is preserved by $\pi,\pi_{MR}$, so part (i) for $N$ follows.

Assuming also that $\ell=0$, so $\bar{E}$ is generated by $\bar{\tau}$,
then standard arguments show that $\bar{E}$ is just the $(\bar{\kappa},\bar{\tau})$-extender
derived from $i^\Tt$, and therefore that in fact $\bar{E}\in\es^{\bar{M}}$.
But this reflects back to $N$, giving part (ii).
\end{proof}

\begin{proof}[Proof of Theorems \ref{thm:strong_extender_in_sequence} and \ref{thm:proj_to_strength}]
Theorem \ref{tm:ext_tau_a_card} directly implies \ref{thm:strong_extender_in_sequence}.
For \ref{thm:proj_to_strength} note that we may replace the given extender
with a sub-extender derived from a set of form $\tau\cup t$, where $t$ is a finite set of ordinals, and then appeal to \ref{tm:ext_tau_a_card}.
\end{proof}

From \ref{thm:strong_extender_in_sequence}
we immediately get:

\begin{cor}\label{cor:seqdef} Let $N$ be a $(0,\om_1+1)$-iterable premouse and 
$\mu,\delta,\kappa\in N$. Then:
\begin{enumerate}[label=--]
 \item If $N\sats$``$\mu$ is a normal measure'' then $\mu\in\es^N$.
 \item If $N\sats$``$\delta$ is Woodin'' then $N\sats$``$\delta$ is Woodin via extenders in 
$\es^N$''.
 \item If $N\sats\PS+$``$\kappa$ is strong'' then $N\sats$``$\kappa$ is strong via extenders 
in $\es^N$''.
\end{enumerate}
\end{cor}

We next prove a finer variant of Theorem \ref{tm:ext_tau_a_card}.
However, we do not actually need the variant in later sections of the paper.

Recall the Dodd projectum and parameter $\tau_E$ and $t_E$ of a short extender $E$;
see \cite{combin} or \cite[\S2]{extmax} for background.
The most important fact we use in this section regarding this notion
is the following:

\begin{fact}[Steel]\label{fact:Dodd-sound} Let $M$ be a $1$-sound, $(0,\om_1+1)$-iterable 
 premouse which is below superstrong.
 Then every $E\in\es_+^M$ is Dodd-sound.
\end{fact}

\begin{rem}\label{rem:superstrong_Dodd-soundness}
 Note
that \ref{fact:Dodd-sound} and all results in this paper are  for mice with Mitchell-Steel indexing.
A theorem analogous to \ref{fact:Dodd-sound} has been proven by Zeman for mice with Jensen indexing,
without the superstrong restriction (see \cite{zeman_dodd_params}). 
Further,  \ref{fact:Dodd-sound}
 has also been generalized for mice in Mitchell-Steel indexing, to allow extenders of superstrong type  in $\es_+^M$; see \cite{fsfni}.\footnote{For the generalization of the other standard
fine structural facts, such as the solidity of the standard parameter,
the proof ``below superstrong'' adapts to the superstrong case
with very little modification. However, for the proof of Dodd-soundness,
the proof requires significant extra work.}
  With this, Theorem \ref{tm:fine_ext_in_seq}
 should also generalize accordingly in a straightforward manner. However,
 as of at the time of publication of this paper,
 \cite{fsfni} is only at the preprint stage,
 so we have only formally stated \ref{tm:fine_ext_in_seq}
 ``below superstrong''.
 
Note that in \ref{tm:fine_ext_in_seq}, we allow $E$ itself
to ostensibly be of ``superstrong type'', but then it follows
that $t=\emptyset$ and $E\in\es_+^M$, so in fact,
$E$ is not of such type (because $M$ is assumed to be below superstrong).
\end{rem}

\begin{dfn}\label{dfn:amenable_def} 
Let $M$ be a premouse and $E$ a standard short $M$-extender,
with support $\tau+\ell$ where $\tau$ is a limit ordinal and $\ell<\om$. Suppose $E$ is weakly amenable (to $M$). 
Let $U=\Ult_0(M,E)$ (we don't assume $U$ is wellfounded). Let $\kappa=\crit(E)$.
 We say that $E$ is \emph{amenably $\bfrSigma_{m+1}^M$}
 iff $\tau<\rho_0^M$ and $\tau^{+U}$ is wellfounded and $U|\tau^{+U}\sub M$,
 and the standard coding of $E$ as an amenable subset of $U|\tau^{+U}$
 is $\bfrSigma_{m+1}^M$. Here the coding consists of tuples $(\xi,\alpha_\xi,E_\xi)$ where
 $\xi<\kappa^{+M}$ and $E_\xi$ is the natural coding of the extender fragment
 \[ E\rest((M|\xi)\cross[\tau+\ell]^{<\om}) \]
 as a subset of $M|\tau$,
 and $\alpha_\xi$ is the least $\alpha$ such that $E_\xi\in U|(\alpha_\xi+\om)$.
 (By the usual proof (see \cite[\S2]{fsit}), $E_\xi\in U$ and the $\alpha_\xi$'s are cofinal in $\tau^{+U}$.)
  We say that $E$ is \emph{explicitly Dodd-solid above $\tau$}
 iff, letting $j:M\to U$ be the ultrapower map
 and 
\[ t=\{[\{\tau+i\},\id']^{M,0}_E\bigm|i<\ell\},\]
 we have $E_j\rest(\alpha\cup (t\backslash\{\alpha\}))\in U$ for each $\alpha\in t$.
 \end{dfn}

\begin{tm}\label{tm:fine_ext_in_seq}
Let $m<\om$ and let $M$ be an $(m+1)$-sound, $(m,\om_1+1)$-iterable
premouse which is below superstrong. Let $E$ be a standard short $M^\sq$-extender,
with support $\tau+\ell$ where $\tau$ is a limit ordinal and $\ell<\om$. Suppose $E$ is weakly amenable to $M$,
 $\kappa=\crit(E)<\rho_{m}^M$,
$E$ is amenably $\bfrSigma_{m+1}^M$,
$\tau\leq\rho_{m+1}^M$ and $\tau$ is an $M$-cardinal
with $\her_\tau^M\sub U$, where  $U=\Ult_m(M,E)$.
Then:
\begin{enumerate}[label=\tu{(}\roman*\tu{)}]\item 
$U|\tau^{+U}=M||\tau^{+U}$
and \item  if $E$ is explicitly Dodd-solid above $\tau$ then  $E\in\es_+^M$.
\end{enumerate}
\end{tm}

\begin{proof}

Let $j:M\to U$ be the ultrapower map. Let $t$ be as in  Definition \ref{dfn:amenable_def}.
So $E\equiv E_j\rest\tau\cup t$, and we will identify these two extenders in what follows.

If $U||\tau=M||\tau$ then let $\theta=\tau$,
and otherwise let $\lambda$ be least such 
that $U|\lambda\neq M|\lambda$ and let $\theta=\card^M(\lambda)$.
So $\theta$ is an $M$-cardinal and $\theta\leq\tau$, and by condensation and weak amenability, $\kappa^{+M}=\kappa^{+U}
\leq\theta$.
Note that if $E$ is explicitly Dodd-solid above $\tau$ then $E$ is Dodd-sound.
(For suppose $\kappa^{+M}<\tau$.
As $E\rest(\tau\cup t)$ is amenably
$\bfrSigma_{m+1}^M$ and $\tau\leq\rho_{m+1}^M$,
then $E\rest(\alpha\cup t)\in M$
for each $\alpha<\tau$. But $\her_\tau^M\sub U$,
so $E\rest(\alpha\cup t)\in U$.)
Let $e\in[\core_0(M)]^{<\om}$ be such that:
\begin{enumerate}
\item $\theta,\tau\in e$ (recall that $\tau<\rho_0^M$ by \ref{dfn:amenable_def}).
\item If $\core_0(M)$ has largest cardinal $\Omega$ then $\Omega\in e$.
 \item The amenable coding of $E\rest(\tau\cup t)$
 (described in \ref{dfn:amenable_def})
 is $\rSigma_{m+1}^M(\{e\})$.
\item\label{item:theta<tau,lambda} If $\theta<\tau$ then $\chi\in e$
where $\chi<\tau$ is least such that $U|\chi\neq M|\chi$.
\item If $\tau^{+U}<\tau^{+M}$
then $\tau^{+U}\in e$.
\item\label{item:tau^+_match_lambda} If $\tau^{+U}=\tau^{+M}$
but $U|\tau^{+U}\neq M||\tau^{+M}$ then $\lambda\in e$
where $\lambda\in(\tau,\tau^{+U}]$ is least such that $U|\lambda\neq M|\lambda$ (note that in fact, $\lambda<\tau^{+U}$,
because $\tau^{+U}<\OR^U$, because $\kappa^{+M}\leq\tau<\rho_0^M$).
 \item If $E$ is explicitly  Dodd-solid above $\tau$ then there are $a,f\in e$
 such that $a\in[\tau]^{<\om}$ and $[a\cup t,f]^{M,m}_E$ is the (finite) set of Dodd-solidity witnesses (for $t$).
 \item\label{item:dodd-soundness_witness,F}\footnote{This condition is only relevant at the very end of the proof,
 and its motivation will only become clear there; the reader can ignore it until that point.}
 If $E$ is explicitly Dodd-solid above $\tau$ 
 and $\theta=\tau$ and
 \[ \varsigma=\tau^{+U}<\tau^{+M} \]
 and $U|\varsigma=M||\varsigma$ but $M|\varsigma$ is active with an extender $F$ such that $\kappa<\crit(F)$,
 then there are $a,f\in e$ with $a\in[\tau]^{<\om}$ and
 such that
 \[ [a\cup t,f]^{M,m}_E=E\rest(\crit(F)\cup t).\]
 \end{enumerate}

Let $q$
be such that
 $(\om,q)$ is $(m+1)$-self-solid for $M$
and $e\in\rg(\pi)$
where $\Mbar=\cHull_{m+1}^M(\{\pvec_m^M,q\})$ and $\pi:\Mbar\to M$ is the uncollapse ($q$ exists by  \ref{lem:sound_hull}).

Let $\pi(\qbar)=q$, $\pi(\thetabar)=\theta$, etc.
Also write $\tbar$ for the preimage of $t$ ``in the codes''; but note we did not demand that $t\sub\OR^M$.
 So $\Mbar$ is $(m+1)$-sound with $\rho_{m+1}^\Mbar=\om$ and $\qbar=p_{m+1}^{\Mbar}$.
 Let $\Ebar\rest\taubar\cup\tbar$ be defined over $\Mbar$ from $\ebar$
 as $E\rest\tau\cup t$ is defined over $M$ from $e$.
 Then the usual proof that $\Sigma_1$-substructures of premice are premice\footnote{As
 $\Sigma_1$ includes
 a constant symbol for the largest initial segment of the active extender.}
and some similar considerations show that most of the facts reflect to $\Mbar,\Ebar$, etc,
and in particular:
 \begin{enumerate}[label=\arabic*'.,ref=\arabic*']
  \item  $\Ebar\rest\taubar\cup\tbar$ is a weakly amenable short
  $\Mbar$-extender
 with $\kappabar=\crit(\Ebar)<\rho_m^\Mbar$.
\end{enumerate}
 Let $\Ebar$ be the short $\bar{M}$-extender generated by $\Ebar\rest\taubar\cup\tbar$
 and let
\[ \Ubar=\Ult_m(\Mbar,\Ebar).\]
 \begin{enumerate}[resume*]
 \item If $M$ has largest cardinal $\Omega$ then $\Mbar$ has largest cardinal $\bar{\Omega}$.
 \item $\thetabar,\taubar$ are $\Mbar$-cardinals, $\her^\Mbar_{\taubar}\sub\Ubar$
 and $\Mbar|\thetabar=\Ubar|\thetabar$.
 \item If $\theta<\tau$ then $\Mbar|\thetabar^{+\Mbar}\neq\Ubar|\thetabar^{+\Ubar}$,
 and $\bar{\chi}$ is least such that $\Mbar|\bar{\chi}\neq\Ubar|\bar{\chi}$.
 \item If $\theta=\tau$ then:
 \begin{enumerate}[label=--]\item If
  $\tau^{+U}<\tau^{+M}$ then $\taubar^{+\Ubar}<\taubar^{+\Mbar}$
 and $\pi(\taubar^{+\Ubar})=\tau^{+U}$.
 \item   If $\tau^{+U}=\tau^{+M}$ then $\taubar^{+\Ubar}=\taubar^{+\Mbar}$.
 \end{enumerate}
 \item If $\theta=\tau$ then:
  \begin{enumerate}[label=--]

 \item If $U|\tau^{+U}=M||\tau^{+U}$ then $\Ubar|\taubar^{+\Ubar}=\Mbar||\taubar^{+\Ubar}$.

 \item If  $U|\tau^{+U}\neq M||\tau^{+U}$ then
 $\taubar<\bar{\lambda}<\taubar^{+\Ubar}$ and $\bar{\lambda}$ is least
 such that $\Ubar|\bar{\lambda}\neq\Mbar|\bar{\lambda}$.
 \end{enumerate}
 \item If $E$ is explicitly Dodd-solid above $\tau$ then $\Ebar$ is Dodd-solid with respect to $\tbar$.
 That is, 
 for each $\alpha\in\tbar$, we have
 $\Ebar\rest(\alpha\cup(\tbar\backslash(\alpha+1)))\in\Ubar$.
 \item\label{item:dodd-soundness_witness,Fbar} If $E$ is explicitly Dodd-solid above $\tau$
  and $U,M,\varsigma,F$
 are as in condition \ref{item:dodd-soundness_witness,F},
 then $\bar{\varsigma}=\taubar^{+\Ubar}$ and $\Mbar|\bar{\varsigma}$
 is active with $\Fbar$ and the Dodd-soundness witness
 $\Ebar\rest(\crit(\Fbar)\cup\tbar)$ is in $\Ubar$.
\end{enumerate}
(We do not yet know that $\Ubar$ is wellfounded.)
Let $\jbar:\Mbar\to\Ubar$ be the ultrapower map.
Let $\psi:\Ubar\to U$ be the Shift Lemma map.
Define the phalanx $\ph=((\Mbar,m,\thetabar), (\Ubar,m), \thetabar)$.

\begin{clmtwo}\label{clm:escit} $\Ubar$ is wellfounded and $\ph$ is $(\omega_1+1)$-iterable.\end{clmtwo}
\begin{proof}
The argument is mostly similar to that in the proof of \ref{tm:ext_tau_a_card}.
We will lift $m$-maximal trees $\Tt$ on $\ph$ to $m$-maximal trees on 
$M$.\footnote{\label{ftn:superstrong}Because we have assumed that $M$ is below superstrong, we do indeed get $m$-maximal trees on $M$ here. In the version of the proof which should hold without this restriction, we might only get essentially $m$-maximal trees on $M$.} For this we will find embeddings from $\Mbar$ and $\Ubar$ into segments of $M$ with appropriate agreement.
As before, in one case we only see how to find an infinite sequence of embeddings from $\bar{M}$
into various segments of $M$, and use of all these together as base copy maps.
We will initially find such a system of maps inside $U$, and then deduce that there is also such a system in $M$ via the elementarity of $j$.
We first make some general observations that will lead to finding
the system of embeddings in $U$.

Let $R\pins M$. Note that $M$ satisfies condensation
with respect to premice embedded into $R$;
in particular, $M\sats$``For every $s<\om$ and every premouse $S\in R$
such that $S$ is $(s+1)$-sound and $\pi:S\to R$ is $s$-lifting and
and $\crit(\pi)\geq\rho_{s+1}^S$, either (i) $S\pins R$
or (ii) $\alpha\eqdef\crit(\pi)=\rho_{s+1}^S$ and $R|\alpha$ is active
and $S\pins\Ult_0(R|\alpha,F^{R|\alpha})$''.
Therefore $U$ satisfies the same statement regarding its proper segments.

Let  $M_\kappa=\cHull_{m+1}^M(\kappa\cup\{q\})$.
Let $\sigma_\kappa:\Mbar\to M_\kappa$ and $\pi_\kappa:M_\kappa\to M$ be the natural maps
and $\pi_\kappa(q_\kappa)=q$.
Note that  $M_\kappa$ is sound and $M_\kappa\in M$. So $\rho_{m+1}^{M_\kappa}=\kappa\leq\crit(\pi_\kappa)$.
By condensation, $M_\kappa\pins M$.

Now $\tau$ is a $U$-cardinal with $\kappa<\tau\leq j(\kappa)$.
Working in $U$, let
\[ U'=\cHull_{m+1}^{j(M_\kappa)}(\tau\cup\{r\}) \]
with $r\in[\OR]^{<\om}\cap j(M_\kappa)$ chosen such that
$U\sats$``$(\tau,r)$ is $(m+1)$-self-solid for $j(M_\kappa)$
and letting $\varrho_\tau:U'\to j(M_\kappa)$ be the uncollapse,
then $j(q_\kappa),t\in\rg(\varrho_\tau)$''.
Such an $r$ exists by the elementarity of $j$
and by \ref{lem:sound_hull}.
(Note that $U\sats$``$j(M_\kappa)$ is wellfounded'';
the transitive collapse $U'$ is computed inside $U$, where it is well-defined.)
Note that if $\tau=j(\kappa)$ then $t=\emptyset$ and $U'=j(M_\kappa)$ and $r=j(q_\kappa)$.
And if $\tau<j(\kappa)$ then $\rho_{m+1}^{U'}=\tau$, 
so $U'\pins j(M_\kappa)$ by condensation in $U$.
In fact, $U'\pins U|\tau^{+U}$, and we assumed that $U|\tau^{+U}$ is wellfounded,
so $U'$ is wellfounded.

Let $\psi:\Ubar\to j(M_\kappa)$ be the Shift Lemma map
induced by $\sigma_\kappa$ and $\pi$. That is, 
given\footnote{Here for a premouse
$R$, $f^R_{\tau,r}$ is the partial function
$f:\core_0(R)^2\to\core_0(R)$ given by
$f(a',t')=\tau^R(r,a',t')$.}
\[ x=[a\cup \bar{t},f^{\Mbar}_{\tau,z}]^{\bar{M},m}_{\bar{E}}\in\core_0(\Ubar)\]
where $\tau$ is an $\rSigma_m$ term, $z\in\core_0(\Mbar)$
and $a\in[\bar{\tau}]^{<\om}$,
then\footnote{The reader may wonder why the superscript ``$U$''
is placed on the last item in the equation.
This is just because we do not know that $j(M_\kappa)$ is wellfounded, and if it is illfounded, then the meaning of ``$\tau^{j(M_\kappa)}$'' might require further explanation. Since $U\sats$``$j(M_\kappa)$ is a sound premouse'', there is no problem interperting $\tau^{j(M_\kappa)}$ in $U$.}
\begin{eqnarray*} \psi(x)&=&[\pi(a)\cup t,f^{M_\kappa}_{\tau,\sigma_\kappa(z)}]^{M,m}_E\\&=&j(f^{M_\kappa}_{\tau,\sigma_\kappa(z)})(\pi(a)\cup t)\\
&=&(\tau^{j(M_\kappa)}(j(\sigma_\kappa(z)),\pi(a)\cup t))^U.\end{eqnarray*}

Now $\rg(\psi)\sub\rg(\varrho_\tau)$,
for given $x$, etc as above, we have $\pi(a)\sub\tau\sub\rg(\varrho_\tau)$, $t\in\rg(\varrho_\tau)$, and $j(\sigma_\kappa(z))\in\rg(\varrho_\tau)$ since \[ j(\sigma_\kappa(z))\in\rg(j\com\sigma_\kappa)=(\Hull_{m+1}^{j(M_\kappa)}(\{j(q_\kappa)\}))^U\]
and $j(q_\kappa)\in\rg(\varrho_\tau)$. So
$\psi(x)\in\rg(\varrho_\tau)$.

So we can define $\psi':\Ubar\to U'$ by $\psi'=\varrho^{-1}_\tau\com\psi$.
Then $\psi'$ is $m$-lifting,
because if $\varphi$ is $\rSigma_{m+1}$
and $\Ubar\sats\varphi(x)$ then easily
\[ U\sats\text{``}j(M_\kappa)\sats\varphi(\psi(x))\text{''},\]
so $U\sats\text{``}U'\sats\varphi(\psi'(x))\text{''}$,
so $U'\sats\varphi(\psi'(x))$. Also $\psi'\rest\taubar=\pi\rest\taubar$.
And $\psi'$ is c-preserving; if $m=0$ and $M$ has  largest cardinal $\Omega$,
this follows easily from commutativity and the fact that
we put $\Omega\in\rg(\pi)$,
and if $m=0$ and $M$ has no largest cardinal then it is because
then for any $M$-cardinal $\xi$, we have $M|\xi\elem_1 M$ by condensation,
and hence, $\kappa<\max(q)$ (as $\kappa\in\rg(\pi)$),
and so $\Mbar,M_\kappa$ have largest cardinals $\Psi_\om,\Psi_\kappa$
respectively, with $\pi(\Psi_\om)=\pi_\kappa(\Psi_\kappa)=\card^M(\max(q))$.

For $\eta<\theta$,
 let
\[ M_\eta=\cHull_{m+1}^{M}(\eta\cup\{q\}) \]
and $\pi_\eta:M_\eta\to M$ be the uncollapse
and $\sigma_\eta:\Mbar\to M_\eta$ the natural map,
so $\pi_\eta\com\sigma_\eta=\pi$.
Since $\eta<\theta\leq\tau\leq\rho_{m+1}^M$,
we have $M_\eta\in M$. Note that if $\eta$ is an $M$-cardinal
then  $M_\eta$ is $(m+1)$-sound with
 $\eta=\rho_{m+1}^{M_\eta}$
 and $p_{m+1}^{M_\eta}=\sigma_\eta(\qbar)\backslash\eta$,
so $M_\eta\pins M|\theta$.

Now as before, we consider two cases.

\begin{casetwo} $\pi``\thetabar$ is bounded in $\theta$.

Let $\eta=\sup\pi``\thetabar$.
We have $M_\eta$, etc, as above.
Note that either:
\begin{enumerate}[label=--]
\item $\eta$ is a limit cardinal of $M$ (hence the comments above apply), or
 \item $M||\eta$ has largest cardinal $\xi$
 where $\xi$ is an $M$-cardinal and $\xi\in\rg(\pi)$,
 and
 \[ \eta\sub\Hull_{m+1}^M(\xi\cup\{q\})=\Hull_{m+1}^M(\eta\cup\{q\}), \]
 because
 $\rg(\pi)=\Hull_{m+1}^M(\{q\})$ is cofinal in $\eta$; therefore,
 $\rho_{m+1}^{M_\eta}=\xi$ and $p_{m+1}^{M_\eta}=\sigma_\eta(\qbar)\backslash\xi$.
\end{enumerate}
It follows that $M_\eta$ is sound, and  $\crit(\pi_\eta)\geq\eta$.
Since $\eta<\theta\leq\rho_{m+1}^M$, condensation (\cite[Theorem 5.2]{premouse_inheriting})
gives  $M_\eta\pins M|\theta$. Note that
$\sigma_\eta\rest\thetabar=\pi\rest\thetabar$ and
$\sigma_\eta\in M|\theta$.
Since $M|\theta=U|\theta$, therefore $M_\eta\pins U|\theta$ and $\sigma_\eta\in U|\theta$.
Note that $M_\eta\pins\core_0(U')$ as $\eta<\tau$.

Now $\sigma_\eta\rest\thetabar=\pi\rest\thetabar=\psi'\rest\thetabar$
and $\sigma_\eta,\Ubar,U'\in U$, with $\Ubar\in\HC^U$, and moreover,
$U|(\OR^{U'})^{+U}$ is wellfounded. So by absoluteness,
in $U$ there is some c-preserving $m$-lifting embedding
$\widetilde{\psi}:\Ubar\to U'$
with $\widetilde{\psi}\rest\thetabar=\sigma_\eta\rest\thetabar$.

So $U\sats\varphi^+(\Mbar,\Ubar,\thetabar)$,
where $\varphi^+(\Mbar,\Ubar,\thetabar)$ asserts ``There are proper segments $M^*$ and $U^*$ of me,
with $M^*\pins \core_0(U^*)$,
and there are c-preserving $m$-lifting embeddings
\[ \pi^*:\Mbar\to M^*\text{
and }\psi^*:\Ubar\to U^*\]
such that $\pi^*\rest\thetabar=\psi^*\rest\thetabar$
and if $\bar{M}|\bar{\theta}$ has largest cardinal $\bar{\xi}$ then $\rho_{m+1}^{M^*}=\pi^*(\bar{\xi})$''.

So by elementarity, $M\sats\varphi^+(\Mbar,\Ubar,\thetabar)$.
Let $M^*,U^*,\pi^*,\psi^*$ witness this in $M$.
These embeddings are enough to copy $m$-maximal trees on $\ph$
to  $m$-maximal trees on $M$.
Let us point out one detail of the copying process. Suppose $\bar{M}|\thetabar$ has largest cardinal $\xibar$ and let $\xi^*=\pi^*(\xibar)$. Then $\rho_{m+1}^{M^*}=\xi^*$ and
\begin{equation}\label{eqn:where_M^*_is}M|\xi^*\pins M^*\pins U^*|(\xi^*)^{+U^*}=U^*|\psi^*(\bar{\theta}).\end{equation}
When iterating $\ph$, extenders $G$ with $\crit(G)=\xibar$
apply to $\Mbar$.  Let $G^*$ be the lift of $G$.
Then $\crit(G^*)=\xi^*$ and $G^*$ is $U^*$-total. From line (\ref{eqn:where_M^*_is}), it follows that
we can define a copy map
\[ \Ult_m(\Mbar,G)\to i^{U^*}_{G^*}(M^*) \]
in the usual manner. Otherwise the copying is routine.\footnote{If one generalizes the proof by dropping the requirement that $M$ be below superstrong type
(assuming the corresponding generalization of Fact \ref{fact:Dodd-sound} as mentioned in Remark \ref{rem:superstrong_Dodd-soundness}), then the lifted tree might fail the the monotone length condition, but 
it will be essentially $m$-maximal;
this is much as in the proof of  \ref{tm:ext_tau_a_card}.}
\end{casetwo}
\begin{casetwo} $\pi``\thetabar$ is unbounded in $\theta$.

Then $\theta$ is a limit cardinal of $M$,
because $\theta$ is an $M$-cardinal ${\leq\rho_{m+1}^M}$
and there is an
$\bfrSigma_{m+1}^M$-definable cofinal partial map
 $\om\to\sup\pi``\thetabar$.
For each $M$-cardinal $\mu<\theta$ we have $M_\mu,\sigma_\mu\in M|\theta=U|\theta$.
We have $M_\mu,\sigma_\mu,U'\in U|\tau^{+U}$.

Let $C$ be the set of $\Mbar$-cardinals ${<\thetabar}$.
Working in $U$, let $T$ be the tree searching for
$\widetilde{\psi},\widetilde{U}$ and 
a sequence $\left<\widetilde{M}_{\mubar},\widetilde{\sigma}_{\mubar}\right>_{\mubar\in C}$
such that:
\begin{enumerate}[label=--]
 \item $\widetilde{U}\pins U|j(\kappa)^{+U}$
 \item $\widetilde{\psi}:\Ubar\to\widetilde{U}$
 is c-preserving $m$-lifting,
\item for each $\mubar\in C$:
\begin{enumerate}[label=--]\item $\widetilde{U}|\widetilde{\psi}(\mubar)\ins\widetilde{M}_\mubar\pins\widetilde{U}$.
\item $\widetilde{\sigma}_\mubar:\Mbar\to\widetilde{M}_\mubar$ 
is c-preserving $m$-lifting,
 \item $\widetilde{\sigma}_\mubar\rest(\mubar+1)\sub\widetilde{\psi}$.
 \end{enumerate}
\end{enumerate}
We have that $U\sats$``$T$ is illfounded'',
because
$\psi',U',\left<M_\mu,\sigma_\mu\right>_\mu$ exist
and $U|\tau^{+U}$
is wellfounded and models $\ZFC^-$.

Now $T=j(T^M)$ for some $T^M\in M$,
so $M\sats$``$T^M$ is illfounded''.
But then letting $\widetilde{U},\widetilde{\psi},\left<\widetilde{M}_\mubar,\widetilde{\sigma}_\mubar\right>_{\mubar\in C}$
witness this, these objects allow us to lift $m$-maximal trees on $\ph$
to $m$-maximal trees on $M$. (Here when we use an extender $G$ with $\crit(G)=\gammabar<\thetabar$,
we apply it to $\Mbar$, and our next lifting map is of the form
\[ \varphi:\Ult_m(\Mbar,G)\to i(\widetilde{M}_\mubar) \]
where $\mubar=\gammabar^{+\bar{M}}$ and
where $i$ is the upper ultrapower map,
and $\varphi$ is defined as usual using $\widetilde{\sigma}_{\mubar}$.)
\end{casetwo}

This completes both cases, and hence, the proof that $\ph$ is iterable.
\renewcommand{\qedsymbol}{$\Box$(Claim \ref{clm:escit})}\qedhere
\end{proof}

We have $\Mbar|\thetabar=\Ubar|\thetabar$. So comparison of $(\ph,\Mbar)$ uses only 
extenders indexed above $\thetabar$. So by the claim, there is a successful such 
comparison $(\Uu,\Tt)$.

\begin{clmtwo}
We have:
\begin{enumerate}
 \item\label{item:no_drops_etc}  $M^\Uu_\infty=M^\Tt_\infty$, $b^\Uu,b^\Tt$ do not 
drop in model or degree, $b^\Uu$ is above $\Ubar$ and  $i^\Uu\com\jbar=i^\Tt$.
 \item\label{item:agreement} $\thetabar=\taubar$, so $\theta=\tau$.
 \item\label{item:further_agmt} $\Ubar|\taubar^{+\Ubar}=\Mbar||\taubar^{+\Ubar}$,
 so $U|\tau^{+U}=M||\tau^{+M}$.
 \item\label{item:Ebar_on} If $E$ is explicitly Dodd-solid above $\tau$ then $\Ebar\in\es_+^{\Mbar}$, so $E\in\es_+^M$.
\end{enumerate}
\end{clmtwo}
\begin{proof} Because $\Mbar$ is $(m+1)$-sound and $\rho_{m+1}^\Mbar=\om$, standard 
arguments give part \ref{item:no_drops_etc}.

Part \ref{item:agreement}: Suppose that $\thetabar<\taubar$. Then since 
$\her_\taubar^\Mbar\sub\Ubar$, we have $\thetabar^{+\Ubar}=\thetabar^{+\Mbar}$.
But then since  $b^\Uu$ is above $\Ubar$ and does not drop,
\[ 
\Ubar||\thetabar^{+\Ubar}=M^\Uu_\infty||\thetabar^{+M^\Uu_\infty}=
M^\Tt_\infty||\thetabar^{+M^\Tt_\infty}=\Mbar||\thetabar^{+\Ubar}=
\Mbar||\thetabar^{+\Mbar}, \]
contradicting the choice of $\theta$ (and hence $\thetabar$).

Part \ref{item:further_agmt}: Much as in  part \ref{item:agreement},
but now with $\taubar=\thetabar$, so $\crit(i^\Uu)\geq\taubar$. The conclusion that $U|\tau^{+U}=M||\tau^{+U}$
follows from the reflection between $\Mbar$ and $M$ discussed earlier.

Part \ref{item:Ebar_on}: If $\Ebar\in\es^\Mbar$,
note that $\Ebar\in\es(\core_0(\Mbar))$, since $\taubar<\rho_0^\Mbar$;
it easily follows then that $E=\pi(\Ebar)$, just by the elementarity of $\pi$.
Similarly if $\Ebar=F^\Mbar$ then $E=F^M$ by elementarity.
So we just need to see that $\Ebar\in\es_+^\Mbar$, assuming that $E$ is explicitly Dodd-solid above $\tau$.

If $t=\emptyset$
then this follows from the ISC
as in the proof of the ISC for pseudo-mice.
Suppose instead that $E$ is explicitly Dodd-solid above $\tau$ and $t\neq\emptyset$. 
So as discussed earlier, $\Ebar$ is Dodd-solid with respect to $\tbar$.
Since $\Mbar$ is $1$-sound and iterable, by \ref{fact:Dodd-sound}
and as in \cite[\S2]{extmax}, we can analyse the Dodd-structure of the extenders used in $\Tt$,
decomposing them into Dodd-sound extenders.
As there, there is exactly one extender
$G=E^\Tt_\alpha$ used along $b^\Tt$, $G$ has largest generator
$\gamma=i^\Uu(\max(\tbar))$, and there is a unique $\beta\leq_\Tt\alpha$
such that the Dodd-core $D$ of $G$ is in $\es_+(M^\Tt_\beta)$.
Moreover, $\tau_D\leq\taubar$,
and if $\beta<_\Tt\alpha$,
then letting $\varepsilon+1=\successor^\Tt(\beta,\alpha)$,
we have $M^{*\Tt}_{\varepsilon+1}=M^\Tt_\beta|\lh(D)$ and $\deg^\Tt_{\varepsilon+1}=0$,
and letting $k=i^{*\Tt}_{\varepsilon+1,\alpha}$, then $\crit(k)\geq\tau_D$,
\[ i^\Uu(\tbar)=k(t_D)\backslash\taubar \]
and
\[ \Ebar\rest\taubar\cup\tbar\equiv G\rest\taubar\cup k(t_D). \]
Note that $\rho_1(M^\Tt_\beta|\lh(D))\leq\tau_D\leq\taubar$.

Suppose $D\neq F^{M^\Tt_\beta}$.
Then $\beta=0$, as otherwise $\taubar<\lh(E^\Tt_0)\leq\rho_1(M^\Tt_\beta|\lh(D))$,
contradiction.
So $D\in\es^{\Mbar}$. Since $\taubar$ is an $\Mbar$-cardinal,
therefore $\tau_D=\taubar$, so
\[ G\rest\taubar\cup k(t_D)\equiv D\rest\tau_D\cup t_D, \]
so $\bar{E}=D\in\es^{\Mbar}$, as desired.

Now suppose instead that $D=F^{M^\Tt_\beta}$.
Then again $\beta=0$, since otherwise $\taubar\leq\lambda(E^\Tt_0)<\tau_D$, contradiction.
So $D=F^\Mbar$. We claim that $\alpha=0$, so $G=D$
is Dodd-sound, and it follows then (as in \cite{extmax}) that
$\Uu$ is trivial and we are done.
So suppose $0<_\Tt\alpha$; so $(0,\alpha]_\Tt$ does not drop in model.
Let $F^*$ be the first extender used along $(0,\alpha]_\Tt$.
So $\taubar\leq\nu_{F^*}$,  as $\taubar<\lh(E^\Tt_0)$ and $\taubar$ is an $\Mbar$-cardinal. Note 
\[ \Ubar=\Hull_{m+1}^{M^\Tt_\infty}(\taubar\cup k(t_D))=\Ult_m(\Mbar,G')\]
where $G'=F^{\Ult_0(\Mbar,F^*\rest\taubar)}$.
Therefore $\Ubar$ is the iterate of $\Mbar$ given by the tree $\Tt'$
which uses exactly two extenders,
$E^{\Tt'}_0=F^*\rest\taubar$ and $E^{\Tt'}_1=G'$. It follows that $\Tt=\Tt'$, $\Uu$ is trivial,
$E^\Tt_0=F^*$, $\nu_{E^\Tt_0}=\taubar$,
$\tau_D\leq\crit(E^\Tt_0)<\taubar$
and $E^\Tt_1=G=G'$.
So $E^\Tt_0\neq F^{\bar{M}}$ (as $\kappabar=\crit(E^\Tt_1)=\crit(D)$ and $D=F^{\bar{M}}$),
so
$\lh(E^\Tt_0)=\taubar^{+\Ubar}=\bar{\varsigma}<\taubar^{+\Mbar}$,
$\Mbar|\bar{\varsigma}$ is active with $E^\Tt_0$,
and $\kappabar<\crit(E^\Tt_0)$. It follows that $E^\Tt_0=\Fbar$ from
property \ref{item:dodd-soundness_witness,Fbar} above. But then by that property,
\[ \Ebar\rest(\crit(E^\Tt_0)\cup\tbar)\in\Ubar\cap\Mbar.\]
Also $\tbar=k(t_D\backslash\crit(E^\Tt_0))$
and
\[ D\equiv D\rest(\tau_D\cup t_D)\equiv D\rest(\crit(E^\Tt_0)\cup (t_D\backslash\crit(E^\Tt_0)))
 \equiv\Ebar\rest(\crit(E^\Tt_0)\cup\tbar).
\]
But then $D\in\Mbar$, contradiction.
\end{proof}

This completes the proof of the theorem.
\end{proof}

\section{Inductive condensation stack}\label{sec:con_stack}

In this section we prove Theorem \ref{thm:E_def_from_e}.
We first give the proofs of some older results, as their methods are then used in the proof 
of \ref{thm:E_def_from_e}.
The first is an observation due to Jensen.

\begin{fact}[Jensen]
Let $N$ be a premouse of height $\kappa>\om$,
where $\kappa$ is regular.
Let $P$ be a sound premouse such that $N\ins P$,
$\rho_\om^P=\kappa$, and $\om$-condensation
holds for $P$. Let $Q$ be likewise.
Then $P\ins Q$ or $Q\ins P$.
\end{fact}
\begin{proof}Suppose not. Taking a hull of $V$, it is easy to 
find $\bar{P},\bar{Q}$ such that $\bar{P}\nins\bar{Q}\nins\bar{P}$
and fully elementary  maps
$\pi:\bar{P}\to P$ and $\sigma:\bar{Q}\to Q$
and $\bar{\kappa}$
such that
\[ \crit(\pi)=\bar{\kappa}=\crit(\sigma)=\rho_\om^{\bar{P}}=\rho_\om^{\bar{Q}}<\kappa \]
and $\pi(\bar{\kappa})=\kappa=\sigma(\bar{\kappa})$.
So by condensation, either
\begin{enumerate}[label=\tu{(}\roman*\tu{)}]
 \item $\bar{P}\ins  N$
and $\bar{Q}\ins  N$, or
 \item$N|\bar{\kappa}$ is active
and $\bar{P}\ins U$ and $\bar{Q}\ins U$ where $U=\Ult(N|\bar{\kappa},F^{N|\bar{\kappa}})$.
\end{enumerate}
In either case, it follows that either $\bar{P}\ins\bar{Q}$
or $\bar{Q}\ins\bar{P}$, a contradiction.
\end{proof}

A slight adaptation gives:
\begin{fact}\label{fact:Jensen_regular}Let $M$ be a $(0,\om_1+1)$-iterable premouse with no 
largest proper segment.
Let $\kappa>\om$ be a regular cardinal of $M$.
Let $P\in M$ be a sound premouse such that $M|\kappa\ins P$,
$\rho_\om^P=\kappa$, and $\om$-condensation
holds for $P$. Then $P\pins M$.\end{fact}
\begin{proof}
Use the proof above with $Q\ins M$ such that $P\in Q$ and $\rho_\om^Q=\kappa$.
\end{proof}

A slight refinement of this argument gives:

\begin{fact}\label{fact:Jensen_regular_lgst_proper_seg} Let $M$ be a $(0,\om_1+1)$-iterable premouse.
Let $\kappa>\om$ be a regular cardinal of $M$.
Let $P\in M$ be a $(n+1)$-sound premouse such that $M|\kappa\ins P$,
$\rho_{n+1}^P=\kappa$, and $(n+1)$-condensation holds for $P$.
Then $P\pins M$.\end{fact}

The second ingredient is an argument of Woodin's,
which is used in the proof of
Corollary \ref{cor:L(pow(kappa))^M_AC} below.
Steel noticed that
 \ref{cor:L(pow(kappa))^M_AC}
follows from Theorem \ref{thm:proj_to_strength} combined with Woodin's argument.

\begin{proof}[Proof of Corollary \ref{cor:L(pow(kappa))^M_AC}]
We have that $M$ is $(0,\om_1+1)$-iterable,
$\kappa$ is an uncountable cardinal in $M$ and $\kappa^{+M}<\OR^M$.
We want to see that 
$M|\kappa^{+M}$ is definable from parameters over $\her=(\her_{\kappa^+})^M$.
There are two cases.

\begin{casethree} $M$ has no cutpoint 
in $[\kappa,\kappa^{+M})$.

Then there are unboundedly many 
$\gamma<\kappa^{+M}$ indexing an $M$-total extender. So by 
\ref{thm:proj_to_strength},
given a premouse $P\in\her$ such that $M|\kappa\ins P$ and $\rho_\om^P=\kappa$, we have
$P\pins M|\kappa^{+M}$
iff there is $E\in\her$ such that $P\pins\Ult(M|\kappa,E)$ 
and $\her\sats$``$E$ is a countably complete short extender''.
This gives a definition of $M|\kappa^{+M}$ over $\her$ from the parameter $M|\kappa$,
which suffices.
\end{casethree}
\begin{casethree} Otherwise ($M$ has a cutpoint $\gamma_0\in[\kappa,\kappa^{+M})$).

The proof in this case is due to Woodin, and was found earlier.
Let $X$ be the set of all $H\in\HC^M$ 
such that there is $P\pins M|\kappa^{+M}$ and $\pi\in M$
such that $\pi:H\to P$ is elementary.
Since $\kappa^{+M}<\OR^M$, we have $X\in M$ and $X$ is essentially a subset of $\om_1^M$ in $M$.
So $X\in\her$.
Let $P\in M$ be a sound premouse such that $M|\gamma_0\ins P$,
$\gamma_0$ is a cutpoint of 
$P$ and $\rho_\om^P\leq\gamma_0$. We claim that the following are equivalent:
\begin{enumerate}[label=(\roman*)]\item\label{item:P_pins_M} $P\pins M$, \item\label{item:countable_submodels_in_X} 
$\her\sats\text{``Every countable elementary 
submodel of 
}P\text{ is in }X\text{''}$.
\end{enumerate}It follows that $M|\kappa^{+M}$ is definable over 
$\her$
from the parameter $(X,M|\gamma_0)$, which suffices.
Now \ref{item:P_pins_M} implies \ref{item:countable_submodels_in_X} by definition. So suppose \ref{item:countable_submodels_in_X} holds. Let $Q$ be such that
$P\in Q\pins M$, with 
$\rho_\om^{Q}\leq\gamma_0$.
Working in $M$, let $Y\elem Q$ be countable, with $P\in Y$.
The transitive collapses $\Pbar$ of $Y\cap P$ and $\Qbar$ of $Y$ are in $X$,
so can be compared in $V$.
But $\Pbar|\gammabar_0=\Qbar|\gammabar_0$ where $\gammabar_0$
is a cutpoint of both $\Pbar,\Qbar$, and $\Pbar,\Qbar$ are sound and project $\leq\gammabar_0$.
So standard calculations give that $\Pbar\ins\Qbar$, so $P\ins Q$.\qedhere\end{casethree}
\end{proof}

Woodin's argument above makes use of the parameter
$X$. We can actually replace this parameter with $\Momone^M$:

\begin{lem}\label{lem:countable_submodels}
Let $N$ be an $(0,\om_1+1)$-iterable premouse with no largest proper segment. Let $M\pins N$ and 
$H\in\HC^N$ and $\pi:H\to M$ be elementary
with $\pi\in N$. Then there is 
$\Mbar\pins N|\om_1^N$ 
and 
an elementary $\pibar:H\to\Mbar$ with $\pibar\in N$.\end{lem}
\begin{proof}
 Let $M\pins P\pins N$ be such that $\pi\in P$.
Let $q\in[\OR^P]^{<\om}$ be 
such that $(\om,q)$ is $1$-self-solid for $P$ and such that
$\pi,H,M\in\Hull_1^P(\{q\})$.
Let
$\Pbar=\cHull_1^P(\{q\})$.                                                                                                                                                                                                                                                                                                                                                                                                                                                                                                                                                                                                                                                                                                                                                                                                                                                                                                                                                                                                                                                                                                                                                                                                                                                                                                                                                                                                                                                                                                                                                                                                                                                                                                                                                                                                                                                                                                                                                                                                                                                                                                                                                                                                                                                                                                                                                                                                                                                                                                                                                                                                                                                                                                                                                                                                                                                                                                                                                                                                                                                                                                                                                                                                                                                                                                                                                                                                                                                                                                                                                                                                                                                                                                                                                                                                                                                                                                                                                                                                                                                                                                                                                                                                                                                                                                                                                                                                                                                                                                                                                                                                                                                                                                                                                                                                                                                                                                                                                                                                                                                                                                                                                                                   
 Then by \ref{lem:sound_hull}, $\Pbar\pins N|\om_1^N$. Let $\sigma:\Pbar\to P$ be the uncollapse.
 Then $\sigma(H)=H$. Let $\sigma(\pibar)=\pi$ and $\sigma(\Mbar)=M$. Then $\Mbar\pins\Pbar$ and 
$\pibar:H\to\Mbar$ elementarily, so we are done.
\end{proof}

Similarly:
\begin{lem}\label{lem:finer_countable_submodels}
Let $N$ be a $(0,\om_1+1)$-iterable premouse. Let $M\pins N$ and
$H\in\HC^N$ and $m<\om$ and $\pi:H\to M$ be an $m$-lifting \tu{((}weak, near\tu{)} $m$-embedding 
respectively\tu{)}
with $\pi\in N$. Then there is 
$\Mbar\pins N|\om_1^N$ 
and 
an $m$-lifting \tu{((}weak, near\tu{)} $m$-embedding respectively\tu{)} $\pibar:H\to\Mbar$ with 
$\pibar\in N$.\end{lem}
\begin{proof}
Consider the case that $N=\J(M)$
and $\pi:H\to M$. Then there is $k<\om$ and $x\in M$ such that $\pi$ is $\rSigma_k^M(\{x\})$. Argue 
as in the proof of \ref{lem:countable_submodels}, but at degree $n$ instead of $1$,
with $n>k+m+5$.
\end{proof}

Woodin's argument above is abstracted into the following definition:
\begin{dfn}
 Let $M$ be a $(0,\om_1+1)$-iterable premouse satisfying ``$\om_1$ exists'', with no largest proper 
segment.
 Then $\css^M$ (\emph{countable substructures}) denotes the set of all $H\in\HC^M$ such that for 
some $Q\pins M$,
 there is $\pi\in M$ such that $\pi:H\to Q$ is elementary. (So by 
\ref{lem:countable_submodels}, $\css^M$ is definable over $\Momone^M$, uniformly in $M$.)
Let $P,Q\in M$ be sound premice. Working in $M$,
say that $Q$ is \emph{$\Momone^M$-verified} 
iff the transitive collapse of every countable elementary substructure of $Q$ is in $\css^M$,
and say that $Q$ is an \emph{$(\Momone^M,P)$-lower part} premouse
iff $P\ins Q$, $P$ is a cutpoint of $Q$, $\rho_\om^Q\leq\OR^P$ and $Q$ is $\Momone^M$-verified.
The stack of all $(\Momone^M,P)$-lower part premice 
$Q\in M$ is denoted $\Lp^M_\Momone(P)$.\end{dfn}

Note that $\Lp^M_\Momone(P)$ is definable over $\univ{M}$ from 
$\Momone^M,P$; the fact that it forms a 
stack follows from the proof of \ref{cor:L(pow(kappa))^M_AC}.

In order to prove Theorem \ref{thm:E_def_from_e}, it easily suffices to prove
that if $M$ is passive, $(0,\om_1+1)$-iterable and satisfies $\ZFC^-+$``$\om_1$ exists'',
then $\es^M$ is definable over $\univ{M}$ from $\es^M\rest\om_1^M$,
uniformly in $M$.
We will in fact prove a stronger fact, Theorem \ref{thm:local_E_def_from_e} below, making do 
with less than $\ZFC^-$.
We may assume that $M$ has a largest cardinal $\theta$.
The proof breaks into different cases, 
depending on the nature of $M$ above $\theta$.
Clearly the cases are not mutually exclusive (Case \ref{case:theta_regular}
is in fact subsumed by Case \ref{case:cof(theta)>om}).

\begin{dfn}Let $M$ be a premouse. Let $\kappa<\theta$ be cardinals of $M$.
We say that $\kappa$ is \emph{$\her_\theta$-strong} in $M$
iff there is $E\in M$ such that $M\sats$``$E$ is a countably complete short extender''
and $\crit(E)=\kappa$ and $\her^M_\theta\sub\Ult_0(M,E)$.
\end{dfn}
\begin{dfn}
 A passive premouse $M$ is \emph{eventually constructible}
 iff $M=\J_\alpha(R)$ for some $R\pins M$ and $\alpha>0$.
\end{dfn}
\begin{rem}
 In the theorem statement below, in each case we specify
 definability classes $\Gamma,\Lambda$.
 The case specification is $\Gamma^{\univ{M}}(\{M|\theta\})$, meaning that there is a $\Gamma$ 
formula $\varphi$
such that for any $(0,\om_1+1)$-iterable premouse $M$ satisfying ``$\om_1$ exists and $\theta$ is 
the largest cardinal'', the case hypothesis holds of $M$ iff $\univ{M}\sats\varphi(M|\theta)$.
In the given case, the definition of $\es^M$ is $\Lambda^{\univ{M}}(\{M|\theta\})$.
(The definability of the case specification is  used in 
defining $M|\theta$ from 
$\Momone^M$ over $\univ{M}$.)
\end{rem}
\begin{dfn}
 Let $M$ be a passive premouse with a largest cardinal $\theta\geq\om_1^M$.
 We say that $M$ is \emph{tractable} iff either
 \begin{enumerate}[label=(\roman*)]
  \item  $\theta$ is regular in $M$,
 or \item $\theta$ is a cutpoint of $M$, or \item  $M$ has no cutpoint in $[\theta,\OR^M)$,
 or \item  $\cof^M(\theta)>\om$, or \item $M\sats$``$\theta$ is not a limit of cardinals which are $\her_\theta$-strong'',
 or \item $\cof^{\bfSigma_2^{\univ{M}}}(\OR^M)>\om$,
 or \item $\cof^{\bfSigma_1^{\univ{M}}}(\OR^M)>\om$ and $M$ is eventually constructible.\qedhere
  \end{enumerate}
\end{dfn}

\begin{tm}\label{thm:local_E_def_from_e}
 Let $M$ be a passive $(0,\om_1+1)$-iterable premouse satisfying ``$\om_1$ exists''.
Then:
\begin{enumerate}[label=\tu{(}\alph*\tu{)}]
 \item\label{item:Sigma_2-cof>om} If $M$ is tractable then
$\es^M$ is 
$\Sigma_4^{\univ{M}}(\{\Momone^M\})$,  uniformly in such $M$.
\item\label{item:ZFC} If $\univ{M}\sats\PS$ then $\es^M$ is $\Sigma_2^{\univ{M}}(\{\Momone^M\})$,
uniformly in such $M$.
\item\label{item:various} Suppose that $M$ has largest cardinal $\theta$ and either:
 \begin{enumerate}[label=\tu{(}\roman*\tu{)},ref=\tu{(}\roman*\tu{)}]
  \item\label{case:theta_regular} $\theta$ is regular in $M$; and let 
$(\Gamma,\Lambda)=(\Pi_1,\Sigma_1)$, or
  \item\label{case:theta_cutpoint} $\theta$ is a cutpoint of $M$; let 
$(\Gamma,\Lambda)=(\Pi_2,\Sigma_2)$, or
  \item\label{case:no_cutpoint} $M$ has no cutpoint in $[\theta,\OR^M)$;
  let $(\Gamma,\Lambda)=(\Pi_3,\Sigma_2)$, or
  \item\label{case:cof(theta)>om} $\cof^M(\theta)>\om$; let 
$(\Gamma,\Lambda)=(\Pi_1,\Sigma_1)$, or
 \item $M\sats$``$\theta$ is not a limit of $\her_\theta$-strong cardinals''; let $(\Gamma,\Lambda)=(\Sigma_3,\Sigma_1)$, or
 \item\label{case:subtle_case} $\cof^{\bfSigma_2^{\univ{M}}}(\OR^M)>\om$;\footnote{By
 $\cof^{\bfSigma_n^{\univ{M}}}(\OR^M)$,
 we mean the least ordinal $\mu$ such that there is a \emph{total}
 unbounded function $f:\mu\to\OR^M$ which is $\bfSigma_n^{\univ{M}}$-definable.
 Note that this is standard $\Sigma_n$, not $\rSigma_n$.}let 
$(\Gamma,\Lambda)=(\Pi_5,\Sigma_4)$, or
  \item\label{case:less_subtle_case}  $\cof^{\bfSigma_1^{\univ{M}}}(\OR^M)>\om$ and $M=\J_\alpha(R)$ for some $R\pins M$ and $\alpha>0$;
let
$(\Gamma,\Lambda)=(\Pi_3\wedge\Sigma_3,\Sigma_3)$.
 \end{enumerate}
  Then $\es^M$ is $\Lambda^{\univ{M}}(\{M|\theta\})$, and the case specification
 is $\Gamma^{\univ{M}}(\{M|\theta\})$, both uniformly in such $M$.
 \end{enumerate}
\end{tm}

\begin{proof}[Proof of Theorem \ref{thm:local_E_def_from_e}]
Parts \ref{item:Sigma_2-cof>om}
and \ref{item:ZFC} follow immediately from part \ref{item:various}
by an easy induction on $M$-cardinals.

Part \ref{item:various}: We split into  cases
corresponding to hypotheses \ref{case:theta_regular}--\ref{case:less_subtle_case}.
In each case we will give a characterization of $\es^M$ and leave to the reader the verification 
of the precise degree of definability. Note that for the definability of the case specification, 
we use \ref{thm:strong_extender_in_sequence} to determine, for 
example,
whether or not $\theta$ is a cutpoint of $M$.

\begin{casefour} $\theta$ is regular in $M$.
 
By \ref{fact:Jensen_regular_lgst_proper_seg},
working in $M$, 
given any premouse $P$, we  have $P\pins M$
iff there is a sound premouse $Q$ and $n<\om$ such that
$P\pins Q$ and $\rho_{n+1}^Q=\theta$
and $M|\theta\ins Q$ and $Q$ satisfies $(n+1)$-condensation.
And $\es^M$ is the stack of all structures of the form $\Ss_m(P)$
for such $P$ and $m<\om$.
\end{casefour}

\begin{casefour} $\theta$ is a cutpoint of $M$.\footnote{
The case specification is $\Pi_2$ because $\theta$ is a cutpoint of $M$
iff for all $E,\her\in M$, if $M\sats$``$\her=\her_{\theta}^M$
and $E$ is a pre-extender with $\her\sub\Ult(M,E)$''
then $M\sats$``$E$ is not countably complete'';
if $\univ{M}$ is admissible then $\Pi_1$ suffices for the case specification,
because we can replace the requirement that $M\sats$``$E$ is not countably complete''
with the requirement that ``$M\sats\Ult(M|\kappa^{+M},E)$ is illfounded'',
where $\kappa=\crit(E)$.}
 
Use the proof of Corollary \ref{cor:L(pow(kappa))^M_AC}, or an  adaptation thereof if 
$M=\J(R)$ for some $R$,
combined with \ref{lem:countable_submodels} and \ref{lem:finer_countable_submodels}.
\end{casefour}
\begin{casefour}
 $M$ has no cutpoint in $[\theta,\OR^M)$.
 
 Use \ref{thm:proj_to_strength}.
\end{casefour}

\begin{casefour} $\cof^M(\theta)>\om$.

Let $P\in M$ and $n<\om$ be such that $P$ is a sound premouse, $M|\theta\ins P$,
$\rho_{n+1}^P=\theta$, and $P$ satisfies $(n+1)$-condensation. We claim that $P\pins M$; clearly 
this suffices. If $\theta$ is regular in $M$ we can use the proof of Case \ref{case:theta_regular},
so suppose otherwise; in particular, $\theta$ is a limit cardinal of $M$.

We prove that $P\pins M$ using a phalanx comparison.
Let $Q\pins M$ and $x\in Q$ and $m<\om$ be such that $P$ is $\rSigma_m^Q(\{x\})$;
in particular, $\OR^P\leq\OR^Q$. We must show that $P\ins Q$.
Suppose not; note that the fact that $P\nins Q$ is first-order over $Q$ (in the parameter $x$).
So we may assume that $x=\emptyset$ (increasing $m$ if needed).
Let $m+n+5<k<\om$ and let $\bar{Q}=\cHull_{k+1}^Q(\emptyset)$.
Then $\bar{Q}\pins M$. Let $\bar{P}$ be defined over $\bar{Q}$ as $P$ is over $Q$.
Let $\pi:\bar{Q}\to Q$ be the uncollapse, and $\pi(\bar{\theta})=\theta$.
Then $\bar{P}$ is $(n+1)$-sound and $\rho_{n+1}^{\bar{P}}=\bar{\theta}$,
$\bar{Q}$ is $\om$-sound and $\rho_{k+1}^{\bar{Q}}=\om$,
$\bar{P}|\bar{\theta}=\bar{Q}|\bar{\theta}$,  $\bar{\theta}$ is a cardinal of both models,
and $\bar{P}\nins\bar{Q}$.

Define the phalanx
$\ph=((\bar{Q},k,\bar{\theta}),(\bar{P},n),\bar{\theta})$.
By the following claim, a standard comparison argument (comparing $\ph$ with $\bar{Q}$) shows that 
$\bar{P}\ins\bar{Q}$,
  a contradiction, completing the proof.

\begin{clmfour} $\ph$ is $(\om_1+1)$-iterable.\end{clmfour}
\begin{proof}
Let $\sigma:\bar{P}\to P$ be $\pi\rest\bar{P}$.
Then $\sigma\rest\bar{\theta}=\pi\rest\bar{\theta}$, and by the choice of $k$, $\sigma$ is $\Sigma_{n+5}$-elementary, so
\begin{equation}\label{eqn:sigma_p_n+1-pres}\sigma\text{ is }p_{n+1}\text{-preserving and preserves }(n+1)\text{-solidity witnesses}.\end{equation}
Let $\eta=\sup\pi``\bar{\theta}$.
Then $\eta<\theta$ because $\cof^M(\theta)>\om$.
Because $\theta$ is a limit cardinal of $M$,
so is $\eta$.
Let
\[ P'=\cHull_{n+1}^P(\eta\un\{\pvec_{n+1}^P\}) \]
and $\pi':P'\to P$ be the uncollapse.

We claim $P'$ is $(n+1)$-sound, $\rho_{n+1}^{P'}=\eta$
and
$q\eqdef p^{P'}_{n+1}=(\pi')^{-1}(p_{n+1}^P)$.
For
\[ P'=\Hull^{P'}(\eta\cup\{q,\pvec_n^{P'}\}), \]
so $\rho_{n+1}^{P'}\leq\eta$. But $P'|\eta=M|\eta$ and $P'\in M$,
and as $\eta$ is an $M$-cardinal, therefore
$\rho_{n+1}^{P'}=\eta$
and $p_{n+1}^{P'}\leq q$. But
 by line (\ref{eqn:sigma_p_n+1-pres}), $(P',q)$ is $(n+1)$-solid (and $\pi'$ maps the $(n+1)$-solidity witnesses of $(P',q)$ to the $(n+1)$-solidity witnesses of $P$). Therefore $p_{n+1}^{P'}=q$ and $P'$ is $(n+1)$-sound, as desired.
 
So we can apply $(n+1)$-condensation to $\pi':P'\to P$ (by hypothesis on $P$), and note that it follows that $P'\pins 
M|\theta\ins Q\pins M$.

Let $\sigma':\bar{P}\to P'$ be the natural factor map. Then $\sigma'$ is a near $n$-embedding,
and $\sigma'\rest\bar{\theta}=\pi\rest\bar{\theta}$. Using $(\pi,\sigma')$, one can
lift normal trees on $\ph$ to normal trees on $Q$, completing the proof.
\end{proof}
\end{casefour}
\begin{casefour}$M\sats$``$\theta$ is not a limit of cardinals which are 
$\her_\theta$-strong''.

This is almost the same as the previous case.
Everything is identical until defining $\eta$.
 Set $\eta=\pi(\bar{\eta})$
where $\bar{\eta}$ is some
 $\bar{Q}$-cardinal $\bar{\eta}<\bar{\theta}$ such that $\bar{\eta}>\kappa$
for all $\her_\theta$-strong cardinals $\kappa$ of $M$, and the $(n+1)$-solidity witnesses of $\bar{P}$ are in $\Hull_{n+1}^{\bar{P}}(\bar{\eta}\cup\{\pvec_{n+1}^{\bar{P}}\})$. (Note we may again assume that $\theta$ is singular in $M$, and hence a limit cardinal of $M$ and $\bar{\theta}$ is a limit cardinal of $\bar{P}$.
We again get that $\eta$ is an $M$-cardinal,
though in this case it might not be a limit cardinal of $M$.)
We get $\sigma'\rest(\bar{\eta}+1)=\pi\rest(\bar{\eta}+1)$, which, by the choice of $\bar{\eta}$, suffices for iterability.
\end{casefour}

The remaining two cases are more subtle than the previous ones.
We (may) now make the:
\begin{ass}\label{ass:limit_cutpoint} $\theta$ is a singular cardinal of $M$ and $M$ has a cutpoint in 
$[\theta,\OR^M)$.\end{ass}
This must of course
be incorporated appropriately into the
$\Sigma_4(\{M|\theta\})$ (in case \ref{case:subtle_case})
and $\Sigma_3(\{M|\theta\})$ (in case \ref{case:less_subtle_case})
definitions one forms from the arguments to follow.
But given the definability $(\Sigma,\Lambda)$
established for cases \ref{case:theta_regular} and \ref{case:no_cutpoint},
this is no problem. (Note here that in case \ref{case:less_subtle_case},
$M$ \emph{does} have a cutpoint $\geq\theta$, so the $\Pi_3$
complexity of asserting the non-existence of a cutpoint is not
relevant in this case.)

\begin{casefour} $\cof^{\bfSigma_2^{\univ{M}}}(\OR^M)>\om$.

Work in $M$ and let $P$ be a premouse. Say that $P$ is \emph{good}
iff
 $P$ is sound,
 $M|\theta\ins P$ and
 $\rho_\om^P=\theta$.
Say that $P$ is \emph{excellent} iff
 $P$ is good,
 $M$ and $\Lp^M_{\Momone}(P)$ have the same universe, and
 $1$-condensation holds for every $Q\pins\Lp^M_{\Momone}(P)$.

By the case hypothesis,
$M$ has no largest proper segment,
so with Assumption \ref{ass:limit_cutpoint},
it follows that there are cofinally many excellent $N\pins M$.
Therefore it suffices to prove the following claim:

\begin{clmfive} Let $P,Q\in M$ be excellent. Then either $P\ins Q$ or $Q\ins P$.
\end{clmfive}
\begin{proof}
We may assume  $Q\pins M$ and $\OR^Q$ is a cutpoint of $M$, so $\Lp^M_{\Momone}(Q)=M$.
Define $\left<P_n,Q_n\right>_{n<\om}$ as follows.
Let $P_0=P$ and $Q_0=Q$. Given $P_n,Q_n$, let $Q_{n+1}$ be the least $N\pins M$ such that $N$ is 
good, $Q_n\pins N$ and $P_n\in N$.
Given $P_n,Q_{n+1}$, let $P_{n+1}$ be the least $R\pins\Lp^M_{\Momone}(P)$
such that $R$ is good, $P_n\pins R$ and $Q_{n+1}\in R$.

Let $\widetilde{P}=\stack_{n<\om}P_n$ and $\widetilde{Q}=\stack_{n<\om}Q_n$.
Note that 
$\widetilde{P}$ and $\widetilde{Q}$ have the same universe $U$
(but ostensibly may have different extender sequences).
We have $\OR^U<\OR^M$ by our case hypothesis, as 
$\left<P_n,Q_n\right>_{n<\om}$ is $\Sigma_2^{\univ{M}}(\{P,Q\})$.\footnote{
It seems that $\Sigma_1$ is not in general enough,
because to ensure that, for example, $P_n\pins\Lp^M_{\Momone}(P)$,
requires a $\all$-quantifier in order to deal with arbitrary countable
substructures of $P_n$; note that if $\cof^M(\theta)>\om$,
one can dispense with this quantifier, however,
as one can code the substructures via bounded subsets of $\theta$.}
Now $\widetilde{P}$ is definable over $U$ from the parameter $P$,
and likewise $\widetilde{Q}$ over $U$ from $Q$; in fact,
\[ \widetilde{P}=\Lp^U_{\Momone}(P)\text{ and }\widetilde{Q}=\Lp^U_{\Momone}(Q).\]
(Clearly cofinally many segments of $\widetilde{P}$ satisfy the requirements
for premice in $\Lp^U_{\Momone}(P)$; but if $R$ is some premouse satisfying
these requirements then working in $U$, we can run the same proof as before
to see that $R\pins\Lp^U_{\Momone}(P)$.)
Also, $U$ has largest cardinal $\theta$, so $\Lp^M_{\Momone}(P)|\OR^U$
and $M|\OR^U$ are both passive.
So letting $P^+=\J(\widetilde{P})$ and $Q^+=\J(\widetilde{Q})$, we have $P^+\pins\Lp^M_{\Momone}(P)$
and $Q^+\pins M$ and (because $\widetilde{P},\widetilde{Q}$ are definable from parameters over $U$),
\[ \univ{P^+}=\univ{\J(U)}=\univ{Q^+}. \]
Also because $\OR^U$ has cofinality $\om$,
definably over $U$ from parameters, we have
\[ \rho_1^{P^+}=\rho_\om^{\widetilde{P}}=\theta=\rho_\om^{\widetilde{Q}}=\rho_1^{Q^+}.\]

We claim that there is  
an $M$-cardinal $\gamma<\theta$ such that
\begin{equation}\label{eqn:hull_captures} 
H\eqdef\Hull_1^{P^+}(\gamma\un p_1^{P^+})\text{ has the same elements as 
}J\eqdef\Hull_1^{Q^+}(\gamma\un p_1^{Q^+}) \end{equation}
(``$\Hull$'' denotes the uncollapsed hull),
and the transitive collapses 
$\bar{P}^+,\bar{Q}^+$
are $1$-sound and such that $\rho_1^{\bar{P}^+}=\gamma=\rho_1^{\bar{Q}^+}$.
For recalling that $\theta$ is a limit cardinal of $M$, let $\gamma<\theta$ be an $M$-cardinal 
large enough
that, defining $H,J$ as above, we have
\[ \widetilde{P},\widetilde{Q},p_1^{P^+},p_1^{Q^+},w_1^{P^+},w_1^{Q^+}\in H\cap J\] 
(recall $w_1^{P^+},w_1^{Q^+}$ are the $1$-solidity witnesses for $P^+,Q^+$).
Then because $\gamma$ is an $M$-cardinal and $w_1^{P^+}\in H$,
we easily have that $\rho_1^{\bar{P}^+}=\gamma$ and $\bar{P}^+$ is $1$-sound, and likewise for 
$\bar{Q}^+$. And because
\[ \gamma\un\{\widetilde{Q},p_1^{Q^+}\}\sub H \]
and $P^+,Q^+$ have the same 
universe, we have $J\sub H$. Similarly $H\sub J$, giving line (\ref{eqn:hull_captures}).

By $1$-condensation for $P^+,Q^+$ (a requirement of excellence), and because $\rho_1^{\bar{P}^+}=\gamma=\rho_1^{\bar{Q}^+}$
is an $M$-cardinal, we have $\bar{P}^+\pins M$ and $\bar{Q}^+\pins M$.
By line (\ref{eqn:hull_captures}), $\OR^{\bar{P}^+}=\OR^{\bar{Q}^+}$. Therefore $\bar{P}^+=\bar{Q}^+$.
It easily follows that $\widetilde{P}=\widetilde{Q}$, giving the claim.
\end{proof}\renewcommand{\qedsymbol}{}
\end{casefour}

\begin{casefour}$\cof^{\bfSigma_1^{\univ{M}}}(\OR^M)>\om$ and $M$ is eventually constructible.

A simplification of the argument in the previous case shows that the collection of all $R\pins M$
such that $M=\J_\alpha(R)$ for some $\alpha>0$, is $\Pi_2^{\univ{M}}(\{M|\theta\})$.
Regarding the complexity of the case specification,
it is $\Sigma_3^{\univ{M}}$ to assert ``$M$ is eventually constructible'',
as it is equivalent to
\[ \exists x\ \all y\ \exists\beta\in\OR\ [y\in \Ss_\beta(x)] \]
($M$ fails to be eventually constructible iff $M$ is closed under sharps).
\end{casefour}

This completes all cases and hence, the proof of the theorem.
\end{proof}

\begin{dfn}
Let $M$ be a transitive structure.
Let $\Momone\in M$ be a premouse with $\univ{\Momone}=\HC^M$.
The \emph{inductive condensation stack of $M$ above $\Momone$}
is the stack of premice in $M$,
extending $\Momone$, satisfying the inductive definition used in the proof of \ref{thm:local_E_def_from_e}.
\end{dfn}

Of course, the inductive condensation stack  $S$ could have
$\OR^S<\OR^M$. But if $M$ is a $(0,\om_1+1)$-iterable tractable premouse
and $\mathfrak{m}=M|\om_1^M$ then $M=S$.

 \begin{rem}
In Case \ref{case:no_cutpoint} of the preceding proof,
it appeared that we used \ref{thm:proj_to_strength}
for extenders $E$ generated by $\theta\cup t$ for some finite set $t$ of generators
(in order that we can represent arbitrary segments $R\pins M|\theta^{+M}$).
Actually, it suffices to consider only extenders
$E$ such that $\nu_E=\theta$ (and $\her_\theta^M\sub\Ult(M,E)$ etc).
 For we claim that (under the case hypothesis) there are unboundedly many $\beta<\OR^M$
 such that $M|\beta$ is active with an extender $E$ such that $\nu_E=\theta$; clearly this suffices.

 For let $Q\pins M$ be such that $\rho_\om^Q=\theta$
 and let $\alpha$ be least such that $\alpha>\OR^Q$ and $M|\alpha$ is active with extender $F$
 and $\kappa=\crit(F)<\theta$. We claim that $\nu_F=\theta$.
 So suppose that $\theta<\nu_F$.
 Easily by the ISC, $\theta$ is the largest cardinal of $M|\alpha$.
 So $F$ is type 2. Let $E=F\rest\theta$,
 let $U_E=\Ult_0(M,E)$, $U_F=\Ult_0(M,F)$
 and $\pi:U_E\to U_F$ the standard factor map. So $\crit(\pi)$ is the least generator $\gamma$
 of $F$ with $\gamma\geq\theta$.
 
 Suppose $\gamma=\theta$.
 Then $\theta$ is a limit cardinal of $M$ and $U_E$,
 so
 $\pi(\theta)>\theta^{+U_F}=\lh(F)$.
 By the ISC, $\kappa$ is $\her_\xi$-strong in $U_E$ for each $\xi<\theta$.
 Therefore $\kappa$ is $\her_\xi$-strong in $U_F$ for each $\xi<\pi(\theta)$.
 But then  by the ISC, there are unboundedly many $\zeta<\theta^{+U_F}$
 indexing an extender $G$ with $\crit(G)=\kappa$, 
 and since $Q\pins U_F|\theta^{+U_F}$, this contradicts the minimality of $F$.
 
 So $\gamma>\theta$. Because
 $\theta^{+U_E}=\lh(E)<\lh(F)=\theta^{+U_F}$,
we have $\gamma=\lh(E)$ and $\pi(\gamma)=\lh(F)$.
 But $E\in\es^{U_F}$, so by reflection,
 there are unboundedly many $\xi<\lh(E)$ such that $M|\xi$ is active
 with an extender $G$ with $\crit(G)=\kappa$,
 and so the same holds of $\pi(\lh(E))=\lh(F)$,
 again contradicting the minimality of $F$.
 \end{rem}

\begin{rem}\label{rem:extra_case}
Let $M$ be passive, $(0,\om_1+1)$-iterable, satisfying ``$\om_1$ exists''
and $\theta=\lgcd(M)$. We sketch, in a further case,
the identification of $M$ from parameter $M|\theta$ over $\univ{M}$.
However, here we do not 
know whether the case specification itself is uniformly definable
over $\univ{M}$ as above. Say that $M$ 
is 
\emph{$\bfrSigma_1$-bounded}
 iff
$\Hull_1^M(\alpha\un\{x\})$
 is bounded in $\OR^M$ for every $\alpha<\rho_1^M$ and $x\in M$.
 Suppose that $M$ is $1$-sound and $\rho_1^M>\om$, and either
$M$ is eventually constructible or
$M$ is \emph{not} $\bfrSigma_1$-bounded.
Then $M$ is definable from $M|\theta$ over $\univ{M}$.

To see this, we argue much as in the last two cases of \ref{thm:local_E_def_from_e}.
We may make Assumption \ref{ass:limit_cutpoint}.
If $M$ is 
eventually constructible things are easier (using then either the argument
from Case \ref{case:less_subtle_case} of \ref{thm:local_E_def_from_e} if
$\cof^{\bfSigma_1^{\univ{M}}}(\OR^M)>\om$, or
a variant of the argument to follow otherwise),
so we leave this case to the reader, and suppose otherwise. So
$M$ is closed under sharps and has no largest proper segment.
The difference to Case \ref{case:subtle_case} of \ref{thm:local_E_def_from_e} is that now, when we 
define $\widetilde{P},\widetilde{Q}$,
we might have $\univ{M}=\univ{\widetilde{P}}=\univ{\widetilde{Q}}$.
Let $P\in M$ be good (\emph{good} defined as before).
Say that $P$ is \emph{outstanding} iff $P$ satisfies
the conditions of excellence from before, and letting $P^*=\Lp_{\Momone}^M(P)$,
then $P^*$ is $1$-sound, $\rho_1^{P^*}>\om$,
$P^*$ is not $\bfrSigma_1$-bounded, $1$-condensation holds for $P^*$,
and for all $R\ins P^*$, if
\[ \exists\kappa\ \big[\om<\rho\eqdef\rho_1^R=\kappa^{+R}\big] \] then for all sufficiently 
large $\gamma<\rho$,
\[ \cHull_1^R(\gamma\un p_1^{R})\text{ is }1\text{-sound}\]
(so $1$-condensation applies to the uncollapse map). By \ref{lem:hulls_proper_segs}, all 
sufficiently large good $Q\pins M$ are outstanding; we take $Q$ such.

Let $P\in M$ be outstanding. We claim  $\rho_1^{P^*}=\rho_1^{M}$.
For suppose $\rho_1^{P^*}<\rho_1^M$. Let $\alpha\in[\rho_1^{P^*},\rho_1^M)$ be large 
enough
that
\[ H\eqdef\Hull_1^M(\alpha\un\{p_1^M\}) \]
is unbounded in $\OR^M$ (using non-$\bfrSigma_1$-boundedness) 
and $P,p_1^{P^*}\in H$. Then $P'\in H$ for cofinally many $P'\pins P^*$.
For given $\eta_0,\eta_1\in H\cap\OR^M$ such that
there is a good $P'\pins P^*$ with $\eta_0\leq\OR^{P'}$ and $P'\in M|\eta_1$,
then the least good $P''\pins P^*$ such that $\eta_0\leq\OR^{P''}$, is in $H$.
(Recall that $\Lp_{\Momone}^M(P)$
is the stack of all good $Q$ such that $P\ins Q$,
$P$ is a cutpoint of $Q$ and all countable elementary substructures of $Q$ \emph{in $M$} have transitive collapse $\bar{Q}\in\css^{\Momone}$.
But in order to identify the desired $P''$,
it suffices to restrict attention to all countable elementary substructures of $P''$
in $M|(\eta_1+\om^2)$; recall here that $M$ is closed under sharps, so $\eta_1+\om^2<\OR^M$. This is because $P'\in M|\eta_1$, and we can run the argument which shows that $P'\ins P''$ or $P''\ins P'$ working in $M_1|(\eta_1+\om^2)$.) It 
follows that
\[ \Hull_1^{P^*}(\alpha\un\{p_1^{P^*}\})\sub H. \]
But $P^*$ is $1$-sound and $\univ{P^*}=\univ{M}$, so $M=H$,
contradicting the fact that 
$\alpha<\rho_1^M$. So $\rho_1^M\leq\rho_1^{P^*}$
and the converse is likewise.

The rest is much like the last part of the argument used in Case \ref{case:subtle_case},
but we might get $\widetilde{P}=P^*$ and $\widetilde{Q}=M$, in which case there is a wrinkle.
If this occurs, choose $\alpha<\rho_1^M=\rho_1^{P^*}$ such that
\[ P,Q\in\Hull_1^{P^*}(\alpha\un p_1^{P^*})\text{ has same elements as }\Hull_1^{M}(\alpha\un 
p_1^M) \]
by arguing as in the previous paragraph,
and such that the transitive collapses $\bar{P},\bar{Q}$ of the hulls are $1$-sound (using 
\ref{lem:hulls_proper_segs} and excellence if $\rho_1^M=\kappa^{+M}$).
Then by $1$-condensation we get $\bar{P}=\bar{Q}$, so $P=Q$.
\end{rem}

\begin{cor}\label{cor:V=HOD}
Let $M$ be a $(0,\om_1+1)$-iterable premouse satisfying either $\PS$ or $\ZFC^-+$``$\om_1$ exists''.
Suppose that either:
\begin{enumerate}[label=\tu{(}\roman*\tu{)}]
\item\label{item:M|om_1^M_iterable} $\Momone^M$ is $(\om,\om_1+1)$-iterable in $M$,\footnote{If 
$\OR^M=\om_2^M$ then this statement should be interpreted as ``There is an $(\om,\om_1)$-strategy 
$\Sigma$ for $M|\om_1^M$ such that for every tree $\Tt$ via $\Sigma$ of length $\om_1$,
there is a $\Tt$-cofinal branch''.} or
\item\label{item:M|om_1^M_constructed} $\Momone^M$  is built by the\footnote{Here
one can naturally impose various other restrictions on the construction,
but it should be uniquely specified somehow.} maximal fully backgrounded 
$L[\es]$-construction of $M$ using background extenders $E\in\es^M$ such that $\nu_E$ is an 
$M$-cardinal.
\end{enumerate}
Then:
\begin{enumerate}
 \item\label{item:es^M_def}  $\es^M$ is definable over 
$\univ{M}$ without parameters,
\item\label{item:M_is_Hull_of_OR} $\univ{M}=\Hull^{\univ{M}}(\OR^M)$, and
 \item\label{item:M_is_OD^M}  if assumption \ref{item:M|om_1^M_iterable} holds and $M\sats\PS$ 
then $\univ{M}=\OD^{\univ{M}}$.\footnote{Recall that
we define $\OD^{\univ{M}}$ as the collection of all $x\in M$
such that $\{x\}$ is definable from ordinal parameters over $\her_\alpha^M$,
for some $\alpha<\OR^M$. So $\OD^{\univ{M}}\sub\Hull^{\univ{M}}(\OR^M)$.}
\end{enumerate}

\end{cor}
\begin{proof}
If assumption \ref{item:M|om_1^M_iterable} holds,
then all three conclusions follow easily from \ref{thm:local_E_def_from_e}.

If assumption \ref{item:M|om_1^M_constructed} holds, by \ref{thm:strong_extender_in_sequence},
if $E\in M$ then 
[$E\in\es^M$ and $\nu_E$  is an $M$-cardinal]
  iff $M\sats\text{``}E$ is a countably complete 
extender, $\nu_E$ is a cardinal and $\her_{\nu_E}\sub\Ult(V,E)$''.
So the 
$L[\es]$-construction using these background extenders is definable over $\univ{M}$ without 
parameters, so $\{\Momone^M\}$ is likewise definable,
so conclusions \ref{item:es^M_def} and \ref{item:M_is_Hull_of_OR}
follow easily from \ref{thm:local_E_def_from_e}.\footnote{Maybe
the $L[\es]$ construction takes $\OR^M$ stages to construct
$\Momone^M$, in which case it's not clear
that the conclusion in clause \ref{item:M_is_OD^M} holds.}
\end{proof}

Recall that $M_\wlim$ is the least proper class mouse with a Woodin limit of Woodins.
Part \ref{item:M|om_1^M_constructed} of the previous corollary gives:
\begin{cor} $\univ{M_\wlim}\sats$``$V=\HOD$''.\end{cor}

There are of course many variants of this corollary. Using the background construction of \cite{premouse_inheriting}
in place of the background construction used above,
one gets that $\univ{M}\sats$``$V=\HOD$''
where $M$ is, for example, the least proper class mouse with a $\lambda$
which is a limit of Woodins and strong cardinals.

\section{Direct condensation stack in $M[G]$}\label{sec:direct_con_stack}

In this section we prove the following theorem, using
a variant of the inductive condensation stack:
\begin{tm}\label{thm:easy_con}
 Let $M$ be a $(0,\om_1+1)$-iterable premouse satisfying $\PS$.
 Let $\theta<\OR^M$ be a regular cardinal of $M$
 and $\PP\in M|\theta$
 be a poset. Let $G$ be $(M,\PP)$-generic.
 Then $\es^M$ is definable over  $M[G]$
 from the parameter $M|\theta$.
\end{tm}
\begin{proof}
 Work in $M[G]$.
 It suffices to give a definition of $M|\eta^{+M}$
 from the parameter $M|\theta$, uniformly in $M$-regular cardinals
 $\eta\geq\theta$.
 Note that the Jensen stack over $M|\eta$ is exactly $M|\eta^+$,
 and this structure satisfies standard condensation facts. (But if $\eta>\theta$, we don't have the parameter $M|\eta$ available to refer to.)

Say that a premouse $P$ is \emph{excellent} iff
$M|\theta\ins P$, $\OR^P=\eta$, the Jensen stack $P^+$
over $P$ has height $\eta^+$, $P^+$ satisfies standard condensation facts,
and there is $\QQ\in P|\theta$ and a $(P,\QQ)$-generic filter $h$
such that $P^+[h]$ has universe $\her_{\eta^+}$.

Clearly the following claim completes the proof:

\begin{clmsix}$M|\eta$ is the unique excellent premouse.\end{clmsix}

\begin{proof}
Clearly $N=M|\eta$ is excellent,
as witnessed by $\PP,g$.

So let $R$ also be excellent, as witnessed by $\QQ,h$.
Define a sequence $\left<N_n,R_n\right>_{n<\om}$
as follows. Let $N_0=N$ and $R_0=R$.
Given $N_n,R_n$, let $N_{n+1}$ be the least $N'$
such that $N_n\pins N'\pins N^+$
and $\rho_\om^{N'}=\eta$ and $R_n,h\in N'[g]$;
then let $R_{n+1}$ be the least $R'$ such that $R_n\pins R'\pins R^+$
and $\rho_\om^{R'}=\eta$ and $N_{n+1},g\in R'[h]$.
Let $N_\om=\stack_{n<\om}N_n$ and $\widetilde{N}=\J(N_\om)$,
and $R_\om,\widetilde{R}$ likewise. Then
$\widetilde{N}\pins N^+$ and $\widetilde{R}\pins R^+$.
Note that
$N_\om[g]$ and $R_\om[h]$ have the same universe $U$,
and $N_\om,R_\om$ are both definable from parameters
over $U$ (via the Jensen stack). Hence, $\J(N_\om[g])$
and $\J(R_\om[h])$ and $\widetilde{N}[g]$ and $\widetilde{R}[h]$
all have the same universe $\widetilde{U}=\J(U)$.

Now $\widetilde{N},\widetilde{R}$ both satisfy
standard $1$-condensation facts.
Let $\gamma<\theta$ be a cardinal of the models $N,N[g],R[h],R$
such that $\PP,\QQ$ have cardinality
$\leq\gamma$ in $N,R$ respectively.

\begin{sclmsix}
For all $x\in\widetilde{U}$
there is $q\in[\OR^{\widetilde{U}}]^{<\om}$
such that the hulls $H,H',J,J'$ all contain the same ordinals, where
\[ H=\Hull_1^{\widetilde{N}}(\gamma\cup\{q\})
 \text{ and }J=\Hull_1^{\widetilde{R}}(\gamma\cup\{q\}),\]
\[ H'=\Hull_1^{\widetilde{U}}(\gamma\cup\{q,N_\om,g\})
=\Hull_1^{\widetilde{U}}(\gamma\cup\{q,R_\om,h\})=J',\]
and moreover, $\PP\in H$, $\QQ\in J$, $x\in H'=J'$,
and the transitive collapses $C,D$ of $H,J$
respectively are sound.\end{sclmsix}

Assuming the subclaim, let $\pi:C\to H$ and $\sigma:D\to J$ be the uncollapses.
Then by $1$-condensation, $C\pins N|\theta$
and $D\pins R|\theta$, and hence $C=D$ (as $N|\theta=R|\theta$ and $\OR^C=\OR^D$),
and $\pi\rest\OR=\sigma\rest\OR$.
But then $\widetilde{N}=\widetilde{R}$ and $N=R$, as desired.

\begin{proof}[Proof of Subclaim] Use a simple variant of the proof of \ref{lem:sound_hull}
to choose $q$, running an algorithm much as there, but simultaneously
for both models $\widetilde{N},\widetilde{R}$,
and using the ``$\Sigma_1$-definability of the $\Sigma_1$-forcing relation''
to see that $H,H'$ contain the same ordinals (and likewise $J,J'$),
and choosing elements of $q$ large enough to ensure
that $H'=J'$ and $\PP\in H$, etc.

Here are some more details: Given $q\rest i$
and $\gamma_i$ much as in the proof of \ref{lem:sound_hull},
first select some $q'_i$ satisfying the requirements
much as before with respect to $\widetilde{N}$ (hence
with $\gamma_i<q'_i<(\gamma_i^+)^{\widetilde{U}}$),
and with $q'_i$ large enough that
$\PP\cup\{\PP\}\sub$ the relevant hulls 
of $\widetilde{N}$ (note this condition holds
trivially unless $\gamma_i<\theta$) 
and $x,N_\om,R_\om,g,h$ are in the relevant hulls of $\widetilde{U}$.
Then choose $q_i$ with $q'_i<q_i<(\gamma_i^+)^{\widetilde{U}}$
and much as before with respect to $\widetilde{R}$.
In this manner it is easy to arrange that $q_i$ works.  We leave the rest to the reader.

This completes the proof of the subclaim, claim and theorem.
\end{proof}
\renewcommand{\qedsymbol}{}
\end{proof}
\renewcommand{\qedsymbol}{}
\end{proof}

\begin{dfn}\label{dfn:fast_con_stack}
 Let $M$ be a transitive structure satisfying $\PS$.
 Work in $M$.
 Let $P$ be a premouse with $\OR^P$ regular.
 For a regular cardinal $\eta\geq\OR^P$, define \emph{$\eta$-excellent premice}
 (relative to $P,\eta$) as in the proof above
  (there we have $P=M|\theta$).
 The \emph{direct condensation stack
 of $M$ above $P$} is the stack $S$ of all $\eta$-excellent premice,
 for all such $\eta$, as far as this is a well-defined stack.
\end{dfn}

\begin{rem}
As a special case of the previous theorem,
we get a shorter proof that if a mouse $M$ satisfies $\PS$,
then $\es^M$ is definable over $\univ{M}$ from the parameter $M|\om_1^M$.
Note that the proof also easily adapts to the case that $M$
has a largest cardinal $\lambda$, assuming that $\lambda$ is $M$-regular.
However, for the singular case (most importantly $\cof^M(\lambda)=\om$)
we need the earlier methods.
\end{rem}

\section{A simplified fine structure}\label{sec:fs}

In \cite{fsit}, Mitchell-Steel fine structure is introduced,
which makes use of the parameters $u_n$.
We introduce a simplified fine structure here which avoids
the parameters $u_n$, and show that in fact,
the two fine structures are equivalent: we get the same notions of soundness, the same projecta and standard parameters, etc.

\begin{dfn}\label{dfn:Hull_k+1}
Let $N$ be a premouse. Given $X\sub N$, $\Hull_{k+1}^N(X)$ denotes
the substructure of $N$ whose elements are
those $z\in N$ such
that there is $\xvec\in X^{<\om}$ and an $\rSigma_{k+1}$ formula
$\varphi$ such that $z$ is the unique $z'\in N$ such that
$N\sats\varphi(\xvec,z')$. And $\cHull_{k+1}^N(X)$ denotes its transitive collapse,
assuming this is well-defined.
Also let $\Th_{k+1}^N(X)$ be the
$\rSigma_{k+1}$ theory\footnote{That is, the pure theory, in the language of
\cite{fsit}.} of $N$ in parameters in $X$.
\end{dfn}

\begin{dfn}[Minimal Skolem terms]\label{dfn:minterm} Let
$\varphi$ be an
$\rSigma_{k+1}$ formula of
$n+1<\om$ free variables $v_0,\ldots,v_n$. The \emph{minimal Skolem term associated
to $\varphi$} is denoted $\minterm_\varphi$, and
has $n$ variables $v_1,\ldots,v_n$.

Let $R$ be a $k$-sound premouse with $\rho_k^R>\om$. Let
$q\in[\rho_0^R]^{<\om}$ such that 
(i) if $k>0$ then
$R=\Hull_k^R(\rho^R_k\un\{q\})$
and (ii) if $q\neq\emptyset$ then $\rho^R_k\leq\min(q)$.
We define the
partial function
\[ \minterm_{\varphi,q}^R:\core_0(R)^n\to\core_0(R). \]
If $k=0$ then $\minterm_{\varphi,q}^R$ is just the usual
Skolem function associated to $\varphi$ (such that the graph of
$\minterm_\varphi^R$ is uniformly $\rSigma_1^R$), with inputs substituting for $v_1,\ldots,v_n$ and output for $v_0$. (Note $q=\emptyset$ in this
case.)

Suppose $k>0$. Let
$\xvec=(x_1,\ldots,x_n)\in\core_0(R)^n$. If $\core_0(R)\sats\neg\ex v_0\ \varphi(v_0,\xvec)$, then
$\minterm_{\varphi,q}^R(\xvec)$ is undefined.

Suppose $\core_0(R)\sats\ex
v_0\ \varphi(v_0,\xvec)$. Let $\tau_\varphi$ be the basic Skolem term associated
to $\varphi$; see \cite[2.3.3]{fsit}.
Recall that $\tau_\varphi^R(\vec{x})$
is the $<_R$-least $y$ such that $\core_0(R)\sats\varphi(y,\vec{x})$. For $\beta<\rho_k^R$, let
$(\tau_\varphi)^\beta$ be defined over $R$ as in the proof of \cite[2.10]{fsit},
with $q$ specified above. Let $\beta_0$ be the
least $\beta$ such that
$(\tau_\varphi)^\beta(\xvec)$ is defined. Then define
\[ \minterm_{\varphi,q}^R(\xvec)=(\tau_\varphi)^{\beta_0}(\xvec).\qedhere \]
\end{dfn}

\begin{lem}\label{lem:comp}
The graph of $\minterm_{\varphi,q}^R$ is
$\rSigma_{k+1}^R(\{q\})$, recursively uniformly in $\varphi,R,q$ \tu{(}for $R,q$ as in
\ref{dfn:minterm}\tu{)}.

Given $\rSigma_{k+1}$ formulas
$\varphi,\psi_0,\ldots,\psi_{n-1}$, with $\varphi$ of $n$ free
variables and $\psi_i$ of $n_i+1$ free variables, the relation
over $\core_0(R)$,
\begin{multline*}\core_0(R)\sats\varrho(\xvec_0,\ldots,\xvec_{n-1})\iff \\
\big(\all i[\xvec_i\in\dom(\minterm^R_{\psi_i,q})]\big)\wedge \core_0(R)\sats\varphi(\minterm_{\psi_0,q}(\xvec_0),\ldots,\minterm_{\psi_{n-1},q}(\xvec_{n-1}
)),
\end{multline*}
is $\rSigma_{k+1}(\{q\})$, uniformly in $R,q$ as in \ref{dfn:minterm}, and
moreover, there is a recursive
function passing from $\varphi,\psi_0,\ldots,\psi_{n-1}$ to an
$\rSigma_{k+1}$ formula for $\varrho$.

Therefore minimal Skolem terms are effectively closed under composition.
That is, for example, there is a recursive function passing from formulas $\varphi,\psi$, each of two free variables $v_0,v_1$, to
$\varrho$, of the same free variables,
such that for all relevant
$R,q$, we have
\[ \minterm_{\varrho,q}^R=\minterm_{\varphi,q}^R\com\minterm_{\psi,q}^R.\]
Likewise for compositions involving larger numbers of variables.
\end{lem}

Note that the \emph{basic Skolem terms}
referred to in the following lemma were
recalled in Definition \ref{dfn:minterm}.

\begin{lem}\label{lem:minterm} 
Let $R,q$ be as in \ref{dfn:minterm}, $X\sub\core_0(R)$ and
\begin{eqnarray*}
H_1&=&\Hull_{k+1}^R(X\un\{q\}),\\
H_2&=&\big\{\minterm^R_{\varphi,q}(\xvec)\bigm|\varphi\text{ is }\rSigma_{k+1}\
\&\ \xvec\in X^{<\om}\big\},\\
H_3&=&\text{ the closure of }X\un\{q\}\text{ under the basic }
\rSigma_{k+1}\text{-Skolem terms.}
\end{eqnarray*}
Then $H_1=H_2=H_3$.
\end{lem}
\begin{proof}The main thing is to see that $H_3\sub H_2$. For this, see the
proof of \cite[2.10]{fsit}, combined with (for example) the observation that
if $\xvec\in X^{<\om}$ and $y\in H_2$ and $R\sats\ex z\ \big[z<_R
y\wedge\varphi(q,\xvec,z)\big]$, then there is $z\in H_2$ such that $z<_R y$ and
$R\sats\varphi(q,\xvec,z)$; this is by \ref{lem:comp}. Applying this
observation finitely many times shows that $\tau^R_\varphi(q,\xvec)\in
H_2$.\end{proof}

\begin{dfn}\label{dfn:u-solidity-witnessed}
For $k<\om$, the terminology \emph{$k$-u-sound}, \emph{$k$-u-universal},  
etc, mean just what \emph{$k$-sound}, \emph{$k$-universal},
etc, mean in \cite[\S2]{fsit}.\footnote{In this notation, ``$k$'' is a variable but
``u'' is just a symbol. The symbol ``u'' indicates that the $u_n$'s are being used in the definition.}
Let $N$ be a $k$-u-sound premouse. We define $\rho_k^N=\rho_k(N)$, $p_k(N)=p_k^N$, $u_k(N)=u_k^N$, $\rho_{k+1}^N=\rho_{k+1}(N)$ and $p_{k+1}(N)=p_{k+1}^N$ as in \cite[Definition 2.8.1]{fsit}.
We also define $\widetilde{p}_0(N)=\widetilde{p}_0^N=\emptyset$
and if $k>0$, define $\widetilde{p}_k(N)=\widetilde{p}^N_k=r$
where $p_k^N=(r,u_{k-1}^N)$.
Given $u\in\core_0(N)$, $\widetilde{p}\in[\rho_0^N]^{<\om}$
and $\alpha\in\widetilde{p}$, define the \emph{$(k+1)$-solidity-witness for $(N,\widetilde{p},u)$ at $\alpha$}, denoted $W_{k+1}^N(\widetilde{p},u,\alpha)$, by
\begin{equation}\label{eqn:W_k+1} W_{k+1}^N(\widetilde{p},u,\alpha)=\Th_{k+1}^N(\alpha\cup\{\widetilde{p}\backslash(\alpha+1),u\}).\end{equation}
Say  $N$ is \emph{$(k+1)$-u-solidity-witnessed}, or just \emph{$(k+1)$-u-witness-solid}, iff
\[W_{k+1}^N(\widetilde{p}_{k+1}^N,u_k^N,\alpha)\in N\text{ for each }\alpha\in \widetilde{p}_{k+1}^N.\qedhere\]
\end{dfn}
Recall here that in \cite[Definition 2.8.2]{fsit} and \cite[Definition 2.15]{outline}, the adjective \emph{$(k+1)$-solid}, as applied to premice $N$,
means something different to
 \emph{$(k+1)$-u-solidity-witnessed}.

\begin{rem}
 Recall that $p_{1}^N=(\widetilde{p}_1^N,\emptyset)$
 and $\widetilde{p}_1^N\in[\rho_0^N]^{<\om}$.
 Suppose that
 $N$ is $k$-sound,
$p_k^N=(\widetilde{p}_k^N,u)$ 
and $s=\widetilde{p}_k^N\in[\rho_{k-1}^N]^{<\om}$, $\lh(s)=\ell$
and $b_0,\ldots,b_{\ell-1}$
are the $k$-u-solidity witnesses
for $(N,s,u)$ (that is, 
$b_i=W_{k}^N(s,u,\alpha_i)$ for each $i<\ell$, where $s=\{\alpha_0>\alpha_1>\ldots>\alpha_{\ell-1}\}$). Recall that
  $p_{k+1}^N$ has form $(r,u_k^N)$
 where:
 \begin{enumerate}[label=--]\item if $\rho_{k-1}^N<\rho_0^N$ then  $u_k^N=(s,u,b_0,\ldots,b_{\ell-1},\rho_{k-1}^N)$, and
 \item if $\rho_{k-1}^N=\rho_0^N$ then $u_k^N=(s,u,b_0,\ldots,b_{\ell-1})$.
 \end{enumerate}
 (Cf.~\cite[Definition 2.8.1]{fsit}, but note that we are discussing $p_{k+1}$ here, not $p_k$.)
\end{rem}

\begin{dfn}\label{dfn:q-fs}For $N$ a premouse, define $q_k=q_k^N$,
\emph{$k$-q-universality},
\emph{$k$-q-solidity} and
\emph{$k$-q-soundness} for $k\in[0,\om)$, recursively as follows. The definitions are actually made uniformly in premice $N$.

Define $q_0^N=\emptyset$ and say that $N$ is \emph{$0$-q-universal}, \emph{$0$-q-solid} and \emph{$0$-q-sound}.

Suppose $q_0,\ldots,q_k$ have been defined and
$N$ is $k$-q-sound and $k$-u-sound.

Now if $k\geq 1$ then suppose by induction that
\[ \core_0(N)=\Hull_k^N(\rho_k^N\un\{q_k,q_{k-1}\}).\]

Let $q_{k+1}$ be the $<_\lex$-least
$q\in[\OR]^{<\om}$ such that\footnote{In our notation, $\Th_{k+1}$ refers to pure
$\rSigma_{k+1}$ theories, but by \cite[\S2]{fsit}, it would make no difference
in the definition of $q_{k+1}$ (or $\rho_{k+1}$) whether we use pure or
generalized theories.}
\[ \Th_{k+1}^N(\rho_{k+1}^N\un\{q,q_{k}\})\notin\core_0(N). \]

Define the \emph{$(k+1)$th core }$\core_{k+1}(N)$ of $N$ as \[ C=\core_{k+1}(N)=\cHull_{k+1}^N(\rho_{k+1}^N\cup\{q_{k+1},q_k\}) \]
and the \emph{$(k+1)$th core map} $\pi:C\to N$
to be the uncollapse map. It will follow from Theorem \ref{thm:fs} below 
that $C$ is $k$-q-sound and $k$-u-sound and if $k\geq 1$ then \[ \core_0(C)=\Hull_k^{\core_0(C)}(\rho_k^C\cup\{q_k^C,q_{k-1}^C\}).\]
So $q_{k+1}^C$ is defined.
Say $N$ is \emph{$(k+1)$-q-universal}
iff $\pow(\rho_{k+1}^N)\cap C=\pow(\rho_{k+1}^N)\cap N$.

Given $u\in\core_0(N)$ and $q\in[\rho_0^N]^{<\om}$ and 
$\alpha\in q$, define
 $W_{k+1}^N(q,u,\alpha)$ exactly as in line (\ref{eqn:W_k+1}) in Definition \ref{dfn:u-solidity-witnessed}, with $\widetilde{p}=q$.\footnote{$W_{k+1}^N$
 was introduced in \ref{dfn:u-solidity-witnessed}
 assuming that $N$ is $k$-u-sound,
 whereas we are now working under the assumption that $N$ is $k$-q-sound. While we might ostensibly be enlarging the domain of this operator, there is no conflict of definitions in case $N$ is both $k$-u-sound and $k$-q-sound.}

We say $(q,u)$ is
\emph{$(k+1)$-solid for $N$} iff $W_{k+1}^N(q,u,\alpha)\in N$ for each $\alpha\in
q$.
We say $N$ is \emph{$(k+1)$-q-solid} iff $(q_{k+1},q_k)$ is
$(k+1)$-solid for $N$.\footnote{Note that the adjective \emph{$(k+1)$-q-solid}, as applied to premice, is analogous
to \emph{$(k+1)$-u-solidity-witnessed}, but \emph{not} analogous to \emph{$(k+1)$-u-solid};
we do not apply the term \emph{$(k+1)$-u-solid} to premice here.
See Definition \ref{dfn:u-solidity-witnessed}.}
We say
$N$ is \emph{$(k+1)$-q-sound} iff $N$ is $(k+1)$-q-solid and
\[ N=\Hull_{k+1}^N(\rho_{k+1}^N\un\{q_{k+1},q_k\}).\qedhere \]
\end{dfn}

The theorem below establishes the equivalence between standard Mitchell-Steel fine structure
(u-soundness, etc) and the fine structure introduced here (q-soundness, etc).
In part \ref{item:near} we show that the parameters provided by $u^N_k$
\emph{automatically} get into the relevant hulls,
so that the direct placement of the $u^N_k$
in those hulls in \cite{fsit} was superfluous.

\begin{tm}\label{thm:fs}
Let $k<\om$. Let $N$ be a premouse. Then:
\begin{enumerate}
\item\label{item:soundness} $N$ is $k$-q-sound iff $N$ is
$k$-u-sound.
\end{enumerate}
If $N$ is $k$-q-sound and $\om<\rho_k^N$ then:
\begin{enumerate}[resume*]
\item\label{item:lower_parameters} $\widetilde{p}_i^N=q_i^N$ for all $i\leq k$.
\item\label{item:near} Let $X\sub N$,
let $C=\cHull_{k+1}^N(X\un\{q_k^N\})$ and
$\pi:C\to N$ be the uncollapse. Then:  
\begin{enumerate}[label=\textup{(\roman*)},ref=(\roman*)]
\item\label{item:C_is_k-u-sound} $C$ is a $k$-u-sound premouse,\item\label{item:pi_is_near_k-u-embedding}  $\pi$ is a
near $k$-u-embedding\footnote{See \cite[Definition 4.1, Remark 4.3]{outline}.}; therefore if $\rho_k^C<\rho_0^C$ then $\pi(\rho_k^C)\geq\rho_k^N$  and $\pi``\rho_k^C\sub\rho_k^N$,
\item\label{item:solidpres} 
$\pi(W_i^C(q_i^C,q_{i-1}^C,\alpha))=W_i^N(q_i^N,q_{i-1}^N,\pi(\alpha))$
for all $i\in[1,k]$ and
$\alpha\in q_i^C$,
\item\label{item:pi_preserves_u_k} $\pi(u_k^C)=u_k^N\in H$.
\end{enumerate}
\item\label{item:k+1} We have:
\begin{enumerate}[label=--]\item $\widetilde{p}_{k+1}^N=q_{k+1}^N$,\item  $N$ is
$(k+1)$-u-witness-solid iff  $N$ is $(k+1)$-q-solid,\footnote{The lack of symmetry here is just due to a divergence in terminology; see \ref{dfn:u-solidity-witnessed}, \ref{dfn:q-fs}.}\item  $N$ is $(k+1)$-u-universal iff $N$ is $(k+1)$-q-universal, and
\item $N$ is
$(k+1)$-u-sound iff $N$ is
$(k+1)$-q-sound.
\end{enumerate}
\end{enumerate}
\end{tm}
\begin{proof}
We prove the proposition by induction on $k$. For $k=0$ it is easy. Assume
$k>0$ and the lemma holds at all $k'<k$. Parts \ref{item:soundness} and
\ref{item:lower_parameters} are trivial by induction (by part
\ref{item:k+1}). So consider part \ref{item:near}.
Let $H=\rg(\pi)$.
Note
that $X\un\{q_k^N\}\sub H$ and if $\xvec\in H$ and
$y\in\Hull_{k+1}^N(\{\xvec\})$ then
$y\in H$. Now we prove:

\begin{clmseven}\label{clm:params_projecta}
Let $i\leq k$. Then $q_i^N\in H$  and if $i<k$ and $\rho_i^N<\rho_0^N$ then
$\rho_i^N\in H$.
\end{clmseven}
\begin{proof}
We prove the claim by induction on $i$.
It is trivial for $i=0$ and $i=k$
(since $q_0^N=\emptyset$ and we put $q_k^N\in H$ directly). Suppose $0<i<k$ and the claim holds for all $i'<i$.

We show $q_i^N\in H$. So assume $n=\lh(q_i^N)>0$.
Note that $q_i^N$ is the unique $q\in[\OR]^{<\om}$ such
that
\begin{enumerate}[label=\tu{(}\roman*\tu{)}]
 \item\label{item:char_q_i_1}
$N=\Hull^N_i(\min(q)\un\{q,  q_{i-1}^N\})$ and
\item\label{item:char_q_i_2} $(q,q_{i-1}^N)$ is
$i$-solid for $N$ and
\item\label{item:char_q_i_3}  $\lh(q)=n$.
\end{enumerate}
Now each of the statements \ref{item:char_q_i_1}--\ref{item:char_q_i_3} are
$\rSigma_{k+1}^N$ assertions of $q$ (in no parameters,
but the $\rSigma_{k+1}$ formulas
used depend on $(\lh(q_0^N),\ldots,\lh(q_i^N)=n)$), and hence, $\{q_i^N\}$ is $\rSigma_{k+1}^N$, so $q_i^N\in H$.
If $i=1$, this is straightforward. If $i>1$,
then by (another) induction, we can refer to $q_j^N$ for each $j<i$, since we already know that $\{q_j^N\}$ is $\rSigma_{k+1}^N$. And for $i>1$, since $N$ is $(i-1)$-q-sound, 
\[ N=\Hull_{i-1}^N(\rho_{i-1}^N\un\{q_{i-1}^N,q_{i-2}^N\}).\]
 So by
\ref{lem:minterm}, statement \ref{item:char_q_i_1} just asserts ``$\all x\in N\ \big[$there is $\betavec\in\min(q)^{<\om}$ and an $\rSigma_i$ formula $\varphi$ such that
$x=\minterm_{\varphi,(q_{i-1}^N,q_{i-2}^N)}(\betavec,q)\big]$'', which is therefore $\rSigma_{k+1}^N$ (recall $i<k$). So $\{q_i^N\}$
is $\rSigma_{k+1}^N$, as required.

Now if $i<k$ and $\rho_i^N<\rho_0^N$ then by induction, $\rho^N_i$ is the least $\rho$ with
\[ N=\Hull^N_i(\rho\un\{q_i^N,q_{i-1}^N\}).\]
So as above and in the proof of \ref{lem:minterm}, if
$\rho_i^N<\OR^N$ then $\rho^N_i\in H$.\footnote{We don't seem to get an
$\rSigma_{k+1}^N$ definition of $\{\rho^N_{k-1}\}$ which is uniform in $N$ here. But we do get some $\rho\in H$ such
that $N=\Hull^N_{k-1}(\rho\un\{q_{k-1}^N,q_{k-2}^N\})$, and then if $\rho\neq\rho_{k-1}^N$,
we can get a smaller such $\rho'\in H$, and so on. Similarly, we don't seem to get a uniform definition of $\{p_{k-1}^N\}$.}
\end{proof}

\begin{clmseven}\label{clm:rSigma_k+1}
 If $\varphi$ is $\rSigma_{k+1}$ and $\xvec\in H$ and $N\sats\ex
y\ \varphi(\xvec,y)$ then $\ex y\in H$ such that $N\sats\varphi(\xvec,y)$.
\end{clmseven}
\begin{proof}
Let $q=\{q_k^N,q_{k-1}^N\}$. Since $q\in H$ and
$N=\Hull_k^N(\rho_k^N\un\{q\})$, \ref{lem:minterm}
applies and yields the claim.
\end{proof}

We have $H=\Hull^N_i(H)$
for each $i\leq k+1$. Therefore by induction, $C$ is $(k-1)$-sound, $\pi$ is a
near $(k-1)$-u-embedding, and so on. Combined with Claim
\ref{clm:params_projecta}, this also gives that if
$\rho_{k-1}^C<\rho_0^C$ then
$\pi(\rho_{k-1}^C)=\rho_{k-1}^N<\rho_0^N$, and if $\rho_{k-1}^C=\rho_0^C$ then
$\rho_{k-1}^N=\rho_0^N$. By this and Claim \ref{clm:rSigma_k+1} it is straightforward to see that $\pi``\rho_k^C\sub\rho_k^N$ and if $\rho_k^C<\rho_0^C$ then $\pi(\rho_k^C)\geq\rho_k^N$. Moreover, 
\[ C=\Hull_k^C(\rho_k^C\un\{\pi^{-1}(q_k^N),q_{k-1}^C\})\]
and $\pi$ is
$\rSigma_{k+1}$ elementary. To see that $q_k^C=\pi^{-1}(q_k^N)$, we
therefore just need that $(\pi^{-1}(q_k^N),q_{k-1}^C)$ is $k$-q-solid for $C$. For
this it suffices to know that $C$ has the appropriate generalized solidity
witnesses; see \cite[\S1.12]{imlc}. But this follows from the fact that $N$ has
generalized solidity witnesses for $(q_k^N,q_{k-1}^N)$ in $\rg(\pi)$, which
follows from Claim \ref{clm:rSigma_k+1}.

We have now established that $C$ is
$k$-q-sound, so by induction, $C$ is $k$-u-sound, giving \ref{item:near}\ref{item:C_is_k-u-sound}. For part \ref{item:near}\ref{item:pi_is_near_k-u-embedding},  it just remains to see that $\pi(p_k^C)=p_k^N$. (We already know that $C,N$ are $k$-u-sound, $\pi$ is an  $\rSigma_{k+1}$-elementary near $(k-1)$-u-embedding which preserves $\rho_{k-1}$, and $\pi``\rho_{k}^C\sub\rho_k^N$.)
Recall that $p_k^C=(\widetilde{p}_k^C,u_{k-1}^C)$
and $p_k^N=(\widetilde{p}_k^N,u_{k-1}^N)$.
By induction, we have $\widetilde{p}_k^C=q_k^C$ and $\widetilde{p}_k^N=q_k^N$, and since $\pi(q_k^C)=q_k^N$, 
we have $\pi(\widetilde{p}_k^C)=\widetilde{p}_k^N$.
So we just need  $\pi(u_{k-1}^C)=u_{k-1}^N$. For this, the only condition that remains to be verified is that $\pi$ preserves  $(k-1)$-u-solidity witnesses. But this follows immediately from $\rSigma_{k+1}$-elementarity
(in fact $\rSigma_k$-elementarity) together with preservation of the earlier defined objects (from the fine structural recursion).
The (full)  $\rSigma_{k+1}$-elementarity of $\pi$ and preservation of the $q_i$'s for $i\leq k$
similarly gives  \ref{item:near}\ref{item:solidpres}, and
part \ref{item:near}\ref{item:pi_preserves_u_k} is likewise.

Part \ref{item:k+1} follows easily from part \ref{item:near},
according to which,
 $u_k^N$ is automatically in
the relevant hulls.
\end{proof}

\subsection*{Acknowledgements} 
The author acknowledges TU Wien Bibliothek for financial support through its Open Access Funding Programme.

The author thanks John Steel for his hospitality during the author's stay at UC Berkeley in 2012/13,
and Steel's NSF grant, for supporting the author during that stay, when the work in \S\ref{sec:fs}
was done.

The author thanks the anonymous referee for their work and various corrections and suggestions for improvements.

The author thanks the \emph{Institut f\"ur Mathematische Logik und Grundlagenforschung,
Universit\"at M\"unster},
for the opportunity to present the inductive condensation stack argument in \S\ref{sec:con_stack} at the
\emph{Oberseminar f\"ur Mengenlehre} in Spring 2016, and also the organizers of the \emph{1st Irvine conference on descriptive inner model theory
and hod mice} (Irvine, July 2016)
for the opportunity to present it at that conference.

\bibliographystyle{plain}
\bibliography{the_bibliography}

Author email:\\  afirstname.alastname at gmail.com, afirstname.alastname at tuwien.ac.at

\end{document}